\documentclass[11pt]{article}
\usepackage{makeidx}         
\usepackage{graphicx}        
\usepackage{psfrag}                             
\usepackage{multicol}        
\usepackage{amsmath}
\usepackage{amssymb}
\usepackage{amsfonts}
\usepackage{latexsym}
\usepackage[T1]{fontenc}
\usepackage{subfigure}                      
\usepackage{float}

\usepackage{amsthm}

\usepackage{color} 
\usepackage{comment}
\newtheorem{lemma}{LEMMA}[section]
\newtheorem{theorem}[lemma]{THEOREM}

\newtheorem{corollary}[lemma]{COROLLARY}
\newtheorem{remark}[lemma]{REMARK}

 \allowdisplaybreaks
 \numberwithin{equation}{section}
\usepackage[colorlinks=true]{hyperref}
\begin{document}

\bgroup


\newcommand{\rd}{\color{red}}

\newcommand{\rc}{\color{red}}

\newcommand{\RD}{\color{red}}

\newcommand{\bl}{\color{blue}}
\newcommand{\cn}{\color{cyan}}
\renewcommand{\cn}{}

\newcommand{\mg}{\color{magenta}}
\renewcommand{\mg}{}
\newcommand{\bn}{\color[rgb]{0.5,0.25,0.0}}

\newcommand{\gr}{\color[rgb]{0.0,0.45,0.0}}
\renewcommand{\gr}{}

\newcommand{\A}{{\cal A}}

\newcommand{\D}{{\cal D}}

\newcommand{\E}{{\cal E}}

\newcommand{\F}{{\cal F}}

\newcommand{\G}{{\cal G}}

\newcommand{\J}{{\cal J}}

\renewcommand{\L}{{\cal L}}

\newcommand{\Q}{{\cal Q}}

\newcommand{\C}{{\mathbb C}}
\newcommand{\R}{{\mathbb R}}
\newcommand{\N}{{\mathbb N}}

\newcommand{\T}{{\mathbb T}}
\newcommand{\Z}{\mathbb{Z}}

\newcommand{\Sc}{{\mathcal S}}

\newcommand{\sZ}{\sum_{\xi\in\Z^n}}
\renewcommand{\div}{{\rm div}}

\newcommand{\bs}{\boldsymbol}

\newcommand{\uxi}{\allmodesymb{\greeksym}{x}}

\newcommand{\mycomment}[1]{}

\allowdisplaybreaks

\title{Stationary Anisotropic Stokes, Oseen, and Navier-Stokes Systems: Periodic Solutions in $\R^n$}




\author{Sergey E. Mikhailov\footnote{
e-mail: {\sf sergey.mikhailov@brunel.ac.uk} 
}\\
Department of Mathematics,
Brunel University London,
UK}

\maketitle

\abstract{First, the solution uniqueness, existence and regularity for stationary anisotropic (linear) Stokes and generalised Oseen systems with constant viscosity coefficients in a compressible framework  are  analysed in a range of periodic Sobolev (Bessel-potential) spaces in $\mathbb R^n$. By the Galerkin algorithm and the Brower fixed point theorem, the existence of solution to the stationary anisotropic (non-linear) Navier-Stokes incompressible system is shown in a periodic Sobolev space for any $n\ge 2$. Then the solution uniqueness and regularity results for stationary anisotropic periodic Navier-Stokes system are established for  $n\in\{2,3,4\}$.
}

\paragraph{MSC-class:}	35J57, 35Q30, 46E35, 76D, 76M
\paragraph{Keywords:} anisotropic Stokes, Oseen, and Navier-Stokes equations; relaxed ellipticity; periodic Sobolev spaces; existence, uniqueness, and regularity;  higher dimensions

\section{Introduction}\label{sec:mik1}
Analysis of Stokes and Navier-Stokes equations is an established and active field of research in the applied mathematical analysis, see, e.g., \cite{Constantin-Foias1988, Galdi2011, RRS2016, Seregin2015, Sohr2001, Temam1995, Temam2001} and references therein.
In \cite{KMW2020, KMW-DCDS2021, KMW-LP2021, KMW-transv2021,KMW2022b} this field has been extended to the transmission and boundary-value problems for stationary Stokes and Navier-Stokes equations of anisotropic fluids, particularly, with relaxed ellipticity condition on the viscosity tensor. 

In this paper, we present some further results in this direction considering space-periodic solutions in $\R^n$, $n\ge 2$, to the stationary Stokes, generalised Oseen and Navier-Stokes equations of anisotropic fluids, with an emphasis on solution regularity.  
The periodic setting is interesting on its own, modelling fluid flow in periodic composite structures, and is also a common element of homogenisation theories for inhomogeneous fluids.
First, the solution uniqueness, existence and regularity for stationary anisotropic (linear) Stokes and generalised Oseen systems with constant viscosity coefficients in a compressible framework  are  analysed in a range of periodic Sobolev (Bessel-potential) spaces on $n$-dimensional flat torus. 
By the Galerkin algorithm, the linear results are employed to show existence of solution to the stationary anisotropic (non-linear) Navier-Stokes incompressible system  on torus in a periodic Sobolev space for any $n\ge 2$.
Then the solution uniqueness and regularity results for stationary anisotropic Navier-Stokes system on torus are established for  $n\in\{2,3,4\}$.
This paper, particularly, extends to the Oseen system and to the Navier-Stokes system for $n>3$ the results obtained in our paper \cite{Mikhailov2022}.

\section{Anisotropic Stokes, Navier-Stokes and Oseen systems}\label{sec:mik2}
Let 
$\boldsymbol{\mathfrak L}$ denote the second-order differential operator represented in the component-wise divergence form as
\begin{align}\label{lllE}
&(\boldsymbol{\mathfrak L}{\mathbf u})_k:=
\partial _\alpha\big(a_{kj}^{\alpha \beta }E_{j\beta }({\mathbf u})\big),\ \ k=1,\ldots ,n,
\end{align}
were ${\mathbf u}\!=\!(u_1,\ldots ,u_n)^\top$, $E_{j\beta }({\mathbf u})\!:=\!\frac{1}{2}(\partial _ju_\beta +\partial _\beta u_j)$ are the entries of the symmetric part ${\mathbb E}({\mathbf u})$ of $\nabla {\mathbf u}$ (the gradient of ${\mathbf u}$),
and $a_{kj}^{\alpha \beta }$ are constant components of the  viscosity coefficient tensor
${\mathbb A}:=\!\left({a_{kj}^{\alpha \beta }}\right)_{1\leq i,j,\alpha ,\beta \leq n}$, cf. \cite{Duffy1978}.
{\mg We also denoted $\partial_j:=\dfrac{\partial}{\partial x_j}$.}
Here and further on, the Einstein summation convention in repeated indices from $1$ to $n$ is used unless stated otherwise.

The following symmetry conditions are assumed (see \cite[(3.1),(3.3)]{Oleinik1992}),
\begin{align}
\label{eq:mik1}
a_{kj}^{\alpha \beta }=a_{\alpha j}^{k\beta }=a_{k\beta }^{\alpha j}.
\end{align}

In addition, we require that tensor ${\mathbb A}$ satisfies the (relaxed) ellipticity condition  in terms of all {\it symmetric} matrices in ${\mathbb R}^{n\times n}$ with {\it zero matrix trace}, see \cite{KMW-DCDS2021,KMW-LP2021}. 
That is, we assume that there exists a constant $C_{\mathbb A} >0$ such that, 
\begin{align}
\label{eq:mik2}
{\mg C_{\mathbb A} a_{kj}^{\alpha \beta }\zeta _{k\alpha }\zeta _{j\beta }\geq |\bs\zeta|^2}\,,
\ \ &\forall\ \bs\zeta =(\zeta _{k\alpha })_{k,\alpha =1,\ldots ,n}\in {\mathbb R}^{n\times n}\nonumber\\
&\mbox{ such that }\, \bs\zeta=\bs\zeta^\top \mbox{ and }
\sum_{k=1}^n\zeta _{kk}=0,
\end{align}
where $|\bs\zeta |^2=\zeta _{k\alpha }\zeta _{k\alpha }$, while the superscript $\top $ denotes the transpose of a matrix.

The tensor ${\mathbb A}$  is endowed with the norm
\begin{align}\label{TensNorm}
\|{\mathbb A}\|:=\max\left\{|a_{kj}^{\alpha \beta }|:k,j,\alpha ,\beta =1\ldots ,n\right\}\,.
\end{align}
Symmetry conditions \eqref{eq:mik1} lead to the following equivalent form of the operator $\boldsymbol{\mathfrak L}$
\begin{equation}
\label{eq:mik3}
\begin{array}{lll}
(\boldsymbol{\mathfrak L}{\mathbf u})_k=\partial _\alpha\big(a_{kj}^{\alpha \beta }\partial _\beta u_j\big),\ \ k=1,\ldots ,n.
\end{array}
\end{equation}

Let ${\mathbf u}$ be an unknown vector field, $p $ be an unknown scalar field, ${\mathbf f}$ be a given vector field and $g$ be a given scalar field.
Then the linear equations
\begin{equation}
\label{eq:mik5}
\begin{array}{lll}
-\boldsymbol{\mathfrak L}{\mathbf u}+\nabla p={\mathbf f},\ {\rm{div}}\ {\mathbf u}=g 
\end{array}
\end{equation}
determine the {\it anisotropic stationary Stokes system with viscosity tensor coefficient ${\mathbb A}=\left(A^{\alpha \beta }\right)_{1\leq \alpha ,\beta \leq n}$ in a compressible framework}.

The nonlinear system
\begin{align}
\label{eq:mik6}
-\boldsymbol{\mathfrak L}{\mathbf u}+\nabla p+{({\mathbf u}\cdot \nabla ){\mathbf u}}={\mathbf f}\,, \ \ {\rm{div}} \, {\mathbf u}=g  
\end{align}
is called the {\it anisotropic stationary Navier-Stokes system with viscosity tensor coefficient ${\mathbb A}\!=\!\left(A^{\alpha \beta }\right)_{1\leq \alpha ,\beta \leq n}$ in a compressible framework}.

In addition, we will also consider a linearised version of the anisotropic stationary Navier-Stokes system \eqref{eq:mik6},
\begin{align}
\label{eq:mik6T}
-\boldsymbol{\mathfrak L}{\mathbf u}+\nabla p+{({\mathbf U}\cdot \nabla ){\mathbf u}}={\mathbf f}\,, \ \ {\rm{div}} \, {\mathbf u}=g,
\end{align}
where $\mathbf U$ is a given function.
We will call \eqref{eq:mik6T} the stationary {\mg anisotropic} Oseen system.
If ${\mathbf U}\equiv 0$, then \eqref{eq:mik6T} reduces to the Stokes system  \eqref{eq:mik6}.

If $g=0$ in \eqref{eq:mik5}, \eqref{eq:mik6}, and \eqref{eq:mik6T}, then these equations are reduced, respectively, to the {\it incompressible anisotropic stationary Stokes, Navier-Stokes, and Oseen systems}.

In the {\it isotropic case}, the tensor ${\mathbb A}$ reduces to
\begin{align}
\label{eq:mik7}
a_{kj}^{\alpha \beta}=\lambda \delta _{k\alpha }\delta _{j\beta }+\mu \left(\delta_{\alpha j}\delta _{\beta k}+\delta_{\alpha \beta }\delta _{kj}\right),\ 1\leq i,j, \alpha ,\beta \leq n\,,
\end{align}
where $\lambda$ and $\mu$ are real constant parameters with
$\mu>0$ (cf., e.g., Appendix III, Part I, Section 1 in \cite{Temam2001})
{\mg and $\delta _{k\alpha }$ is the Kronecker symbol.
Then} \eqref{eq:mik3} becomes
\begin{equation}
\label{eq:mik8}
\boldsymbol{\mathfrak L}{\mathbf u}
=(\lambda +\mu)\nabla{\rm div}\,\mathbf u+\mu\Delta \mathbf u.
\end{equation}
Then it is immediate that condition \eqref{eq:mik2} is fulfilled (cf. \cite{KMW-LP2021}) and thus our results apply also to the Stokes, Navier-Stokes, and Oseen systems in the {\it isotropic case}. 
Assuming $\lambda=0$, $\mu=1$ we arrive at the classical mathematical formulations of isotropic Stokes, Navier-Stokes, and Oseen systems.

\section{Some periodic function spaces}\label{sec:mik3}

Let us introduce some function spaces on torus and periodic function spaces (see, e.g., 
\cite[p.26]{Agmon1965},  \cite{Agranovich2015}, \cite{McLean1991}, \cite[Chapter 3]{RT-book2010}, \cite[Section 1.7.1]{RRS2016}  
\cite[Chapter 2]{Temam1995}, 
for more details).

Let $n\ge 1$  be an integer and $\T$ be the $n$-dimensional flat torus that can be parametrized as the semi-open cube $\T= [0,1)^n\subset\R^n$, cf.  \cite[p. 312]{Zygmund2002}.
In what follows, ${\mathcal D}(\T)=\mathcal C^\infty(\T)$ denotes the space of infinitely smooth real or complex functions on the torus.
As usual, $\N$ denotes the set of natural numbers,  $\N_0$ the set of natural numbers augmented by 0, and $\mathbb{Z}$ the set of integers.

Let   $\bs\xi \in \mathbb{Z}^n$ denote the $n$-dimensional vector with integer components. 
We will further need also the set 
$$\dot\Z^n:=\Z^n\setminus\{\mathbf 0\}.$$
Extending the torus parametrisation to $\R^n$, it is often useful to identify $\T$ with the quotient space $\R^n\setminus \Z^n$. 
Then the space of functions $\mathcal C^\infty(\T)$ on the torus can be identified with the space of $\T$-periodic (1-periodic) functions 
$\mathcal C^\infty_\#=\mathcal C^\infty_\#(\R^n)$ that consists of functions $\phi\in \mathcal C^\infty(\R^n)$ such that
\begin{align}
\mg\label{E3.1}
\phi(\mathbf x+\bs\xi)=\phi(\mathbf x)\quad \forall\,  \bs\xi \in \mathbb{Z}^n,\ \forall\,\mathbf x\in\R^n.
\end{align}
Similarly, the Lebesgue space on the torus $L_{p}(\T)$, $1\le p\le\infty$,  can be identified with the periodic Lebesgue space $L_{p\#}=L_{p\#}(\R^n)$ that consists of functions $\phi\in L_{p,\rm loc}(\R^n)$, which satisfy the periodicity {\mg condition \eqref{E3.1}} (for a.e.~$\mathbf x\in\R^n$).

The space dual  to $\mathcal D(\T)$, i.e.,  the space of linear bounded functionals on $\mathcal D(\T)$,  called the space of torus distributions, is denoted by $\mathcal D'(\T)$ and can be identified with the space of periodic distributions $\mathcal D'_\#$ acting on $\mathcal C^\infty_\#$.

The toroidal/periodic Fourier transform 
mapping a  function $g\in \mathcal C_\#^\infty$ to a set of its Fourier coefficients $\hat g$ is defined as (see, e.g., \cite[Definition 3.1.8]{RT-book2010})  
\begin{align*}
 \hat g(\bs\xi)=[\F_{\T} g](\bs\xi):=\int_{\T}e^{-2\pi i \mathbf x\cdot\bs\xi}g(\mathbf x)d\mathbf x,\quad \bs\xi\in\Z^n.
\end{align*}
and can be generalised to  the Fourier transform acting on a distribution $g\in\mathcal D'_\#$.

For any $\bs\xi\in\Z^n$, let $|\bs\xi|:=(\sum_{j=1}^n \xi_j^2)^{1/2}$ be the Euclidean norm in $\Z^n$ and let us denote $$\rho(\bs\xi):=(1+|\bs\xi|^2)^{1/2}.$$
Evidently,
\begin{align}\label{eq:mik9}
\frac{1}{2}\rho(\bs\xi)^2\le |\bs\xi|^2\le \rho(\bs\xi)^2\quad\forall\,\bs\xi\in \dot\Z^n.
\end{align}

Similar to \cite[Definition 3.2.2]{RT-book2010}, for $s\in\R$ we define the {\em periodic/toroidal Sobolev (Bessel-potential) spaces} $H_\#^s:=H_\#^s(\R^n):=H^s(\T)$, 
which consist of the torus distributions $ g\in\mathcal D'(\T)$, for which the norm
\begin{align}\label{eq:mik10}
\| g\|_{H_\#^s}:=\| \rho^s\widehat g\|_{\ell_2}:=\left(\sZ\rho(\bs\xi)^{2s}|\widehat g(\bs\xi)|^2\right)^{1/2}
\end{align}
is finite, i.e., the series in \eqref{eq:mik10} converges.
Here $\| \cdot\|_{\ell_2}$ is the standard norm in the space of square summable sequences.
By \cite[Proposition 3.2.6]{RT-book2010}, $H_\#^s$ are Hilbert spaces.

The dual product between $g\in H_\#^s$ and $f\in H_\#^{-s}$, $s\in\R$, is defined (cf. \cite[Definition 3.2.8]{RT-book2010}) as
\begin{align}\label{E3.3}
\langle g,f\rangle_{\T}:=\langle \hat g,\hat f\rangle_{\Z^n}:=\sum_{\bs\xi\in\Z^n}\hat g(\bs\xi)\hat f(-\bs\xi).
\end{align}
If $s=0$, i.e., $g,f\in L_{2\#}$, then \eqref{E3.3} reduces to
$$
\langle g,f\rangle_{\T}=\int_{\T}g(\mathbf x)f(\mathbf x)d\mathbf x.
$$
Hence for any $s\in\R$, the space  $H_\#^{-s}$ is adjoint (dual) to  $H_\#^{s}$, i.e., $H_\#^{-s}=(H_\#^{s})^*$.
Similar to, e.g., \cite[p.76]{McLean2000} one can show that 
\begin{align}\label{eq:mik10f}
\| g\|_{H_\#^s}=\sup_{f\in H_\#^{-s}, f\ne 0}\frac{|\langle g,f\rangle_{\T}|}{\|f\|_{H_\#^{-s}}}.
\end{align}

For $g\in H_\#^s$, $s\in\R$, and  $m\in\N_0$, let us consider the partial sums 
$$g_m(\mathbf x)=\sum_{\bs\xi\in\Z^n, |\bs\xi|\le m}\hat g(\bs\xi)e^{2\pi i \mathbf x\cdot\bs\xi}.$$ 
 Evidently, $g_m\in \mathcal C_\#^\infty$, $\hat g_m(\bs\xi)=\hat g(\bs\xi)$ if $|\bs\xi|\le m$ and  $\hat g_m(\bs\xi)=0$ if $|\bs\xi|> m$.
This implies that $\|g-g_m\|_{H_\#^s}\to 0$ as $m\to\infty$ and hence we can write 
\begin{align}\label{eq:mik11}
g(\mathbf x)=\sum_{\bs\xi\in\Z^n}\hat g(\bs\xi)e^{2\pi i \mathbf x\cdot\bs\xi},
\end{align}
where the Fourier series converges in the sense of norm \eqref{eq:mik10}.
Moreover, since $g$ is an arbitrary distribution from $H_\#^s$, this also implies that the space ${\mathcal C}^\infty_\#$ is dense in $H_\#^s$ for any $s\in\R$ (cf. \cite[Exercise 3.2.9]{RT-book2010}).

There holds the compact embedding $H_\#^t\hookrightarrow H_\#^s$ if $t>s$,  embeddings $H_\#^s\subset \mathcal C_\#^m$ if $m\in\N_0$, $s>m+n/2$, and moreover, $\bigcap_{s\in\R}H_\#^s={\mathcal C}^\infty_\#$ (cf. \cite[Exercises 3.2.10, 3.2.10, and Corollary 3.2.11]{RT-book2010}). Note also that the {\cn periodic} norms on $H_\#^s$ are equivalent to the corresponding standard (non-periodic) Bessel potential norms on $\T$ as a cubic domain, see, e.g., \cite[Section 13.8.1]{Agranovich2015}.

By \eqref{eq:mik10}, 
$\| g\|^2_{H_\#^s}=|\widehat g(\mathbf 0)|^2 +| g|^2_{H_\#^s},$ 
where
\begin{align*}
| g|_{H_\#^s}:=\| \rho^s\widehat g\|_{\dot\ell_2}:=\left(\sum_{\bs\xi\in\dot\Z^n}\rho(\bs\xi)^{2s}|\widehat g(\bs\xi)|^2\right)^{1/2}
\end{align*}
is the seminorm in $H_\#^s$.

For any $s\in\R$, let us also introduce  the space
\begin{align}\label{E3.6}
\dot H_\#^s:=\{g\in H_\#^s: \langle g,1\rangle_{\T}=0\}.
\end{align}
The definition implies that if $g\in \dot H_\#^s$, then $\widehat g(\mathbf 0)=0$ and 
\begin{align}\label{eq:mik12}
\| g\|_{\dot H_\#^s}=\| g\|_{H_\#^s}=| g|_{H_\#^s}=\| \rho^s\widehat g\|_{\dot\ell_2}\ .
\end{align}
Denoting 
$
\dot {\mathcal C}^\infty_\#:=\{g\in {\mathcal C}^\infty_\#: \langle g,1\rangle_{\T}=0\}
$,
then $\bigcap_{s\in\R}\dot H_\#^s=\dot {\mathcal C}^\infty_\#$.

Definition \eqref{E3.6} also implies that the space adjoint to $\dot H_\#^s$ can be expressed as the quotient space,
$$
(\dot H_\#^s)^*=(H_\#^s)^*/\R=H_\#^{-s}/\R.
$$
Identifying the quotient space $H_\#^{-s}/\R$ with the space $\dot H_\#^{-s}$, we then identify the space $\dot H_\#^{-s}$ with the space dual to $\dot H_\#^s$.

The corresponding spaces of $n$-component vector functions/{\mg distributions} are denoted as $\mathbf L_{q\#}:=(L_{q\#})^n$, $\mathbf H_\#^s:=(H_\#^s)^n$, etc.

Note that  the norm $\|\nabla (\cdot )\|_{{\mathbf H}_\#^{s-1}}$
is an equivalent norm in $\dot H_\#^s$. 
Indeed, by \eqref{eq:mik11}
\begin{align*}
\nabla g(\mathbf x)=2\pi i\sum_{\bs\xi\in\dot\Z^n}\bs\xi e^{2\pi i \mathbf x\cdot\bs\xi}\hat g(\bs\xi),\quad
\widehat{\nabla g}(\bs\xi)=2\pi i\bs\xi \hat g(\bs\xi)\quad \forall\,g\in \dot H_\#^s,
\end{align*}
and 
then \eqref{eq:mik9} and \eqref{eq:mik12} imply that
\begin{multline}\label{eq:mik13}
2\pi^2\|g\|^2_{H_\#^s}=2\pi^2\|g\|^2_{\dot H_\#^s}=2\pi^2|g|^2_{H_\#^s}\le \| \nabla g\|^2_{{\mathbf H}_\#^{s-1}}\\
\le 4\pi^2|g|^2_{H_\#^s}=4\pi^2\|g\|^2_{\dot H_\#^s}=4\pi^2\| g\|^2_{H_\#^s} \quad \forall\,g\in \dot H_\#^s.
\end{multline}
The vector counterpart of \eqref{eq:mik13} takes the form
\begin{multline}\label{eq:mik14}
2\pi^2\| \mathbf v\|^2_{{\mathbf H}_\#^s}=2\pi^2\| \mathbf v\|^2_{\dot{\mathbf H}_\#^s}
\le \| \nabla \mathbf v\|^2_{(H_\#^{s-1})^{n\times n}}\\
\le 4\pi^2\| \mathbf v\|^2_{\dot{\mathbf H}_\#^s}=4\pi^2\| \mathbf v\|^2_{{\mathbf H}_\#^s} \quad \forall\,\mathbf v\in \dot {\mathbf H}_\#^s.
\end{multline}

We will further need the first Korn inequality
\begin{align}
\label{eq:mik15}
\|\nabla {\bf v}\|^2_{(L_{2\#})^{n\times n}}\leq 2\|\mathbb E ({\bf v})\|^2_{(L_{2\#})^{n\times n}}\quad\forall\, \mathbf v\in {\mathbf H}_\#^1 
\end{align}
that can be easily proved by adapting, e.g., the proof in \cite[Theorem 10.1]{McLean2000} to the periodic Sobolev space; cf. also \cite[Theorem 2.8]{Oleinik1992}.

Let us also define the Sobolev spaces of divergence-free functions and distributions,
\begin{align}\label{E3.8}
\dot{\mathbf H}_{\#\sigma}^{s}
&:=\left\{{\bf w}\in\dot{\mathbf H}_\#^{s}:{\div}\, {\bf w}=0\right\},\quad s\in\R,
\end{align}
endowed with the same norm \eqref{eq:mik10}.
Similarly, ${\mathbf C}^\infty_{\#\sigma}$ and $\mathbf L_{q\#\sigma}$ denote the subspaces of divergence-free vector-functions from 
${\mathbf C}^\infty_{\#}$ and $\mathbf L_{q\#}$, respectively, etc.

\section{Stationary anisotropic periodic Stokes system}\label{sec:mik4}
Let $n\ge 2$. In this section, we generalise to the isotropic and anisotropic (linear) Stokes systems in compressible framework and to a range of Sobolev spaces the analysis available in \cite[Section 2.2]{Temam1995}.

For the unknowns $({\mathbf u},p )\in \dot{\mathbf H}_\#^s\times \dot H_\#^{s-1}$ and the given data
$({\mathbf f},g)\in\dot{\mathbf H}_\#^{s-2}\times \dot H_\#^{s-1}$, $s\in\R$, let us consider the Stokes system
\begin{align}
\label{eq:mik16}
-\bs{\mathfrak L}{\mathbf u}+\nabla p &=\mathbf{f},\\
\label{eq:mik17}
{\rm{div}}\, {\mathbf u}&=g,
\end{align}
that should be understood in the sense of distributions, i.e.,
\begin{align}
\label{eq:mik18}
\langle-\bs{\mathfrak L}{\mathbf u}+\nabla p, \bs\phi\rangle_{\T}&=\langle\mathbf{f}, \bs\phi\rangle_{\T}\quad\forall\,\bs\phi\in {\mg{\mathbf C}^\infty_\#},\\
\label{eq:mik19}
\langle{\rm{div}}\, {\mathbf u}, \phi\rangle_{\T}&=\langle g, \phi\rangle_{\T}\quad\forall\,\phi\in {\mathcal C}^\infty_\#.
\end{align}
For $\bs\xi\in \dot\Z^n$, let us employ  $\bar e_{\bs\xi}(\mathbf x)=e^{-2\pi i x\cdot\bs\xi}$ as $\phi$ in \eqref{eq:mik19} and $\bar e_{\bs\xi}(\mathbf x)$, multiplied by the {\mg unit} coordinate vector, as $\bs\phi$ in \eqref{eq:mik18}. 
Then recalling \eqref{eq:mik3}, we arrive for each $\bs\xi\in \dot\Z^n$ at the following algebraic system for the Fourier coefficients, $\hat u_j(\bs\xi)$, $k=1,2,\dots,n$, and $\hat p(\bs\xi)$.
\begin{align}
\label{eq:mik20}
\hspace{-1.5em}4\pi^2\xi_\alpha a_{kj}^{\alpha \beta}\xi_\beta \hat u_j(\bs\xi) + 2\pi i\xi_k\hat p(\bs\xi)
&=\hat{f}_k(\bs\xi) \quad\forall\,\bs\xi\in \dot\Z^n,\  k=1,2,\dots,n\\
\label{eq:mik21}
2\pi i\xi_j\hat{u}_j(\bs\xi)&=\hat g(\bs\xi) \quad\ \forall\,\bs\xi\in \dot\Z^n.
\end{align}

System \eqref{eq:mik20}-\eqref{eq:mik21} can be written in the form
\begin{align}\label{E4.7}
\mathfrak S(\bs\xi)\left(\widehat{\mathbf u}(\bs\xi) \atop \hat p(\bs\xi)\right)=\left(\widehat{\mathbf f}(\bs\xi) \atop \hat g(\bs\xi)\right) \quad\forall\,\bs\xi\in \dot\Z^n,
\end{align}
where $\mathfrak S(\bs\xi)$ is the $(n+1) \times (n+1)$ matrix with entries 
\begin{align}
\label{ADN-elliptic}
\mathfrak S_{\ell j}(\boldsymbol\xi )=\begin{cases}
4\pi^2\xi _\alpha a_{\ell j}^{\alpha \beta }\xi _\beta,  &  \ell, j=1,\ldots ,n;\\
-2\pi i\xi _\ell, & \ell= 1,\ldots ,n,\ j=n+1;\\
-2\pi i\xi _j,  & \ell=n+1,\ j=1,\ldots ,n;\\
0,& \ell =j=n+1.
\end{cases}
\end{align}
Here $i^2=-1$, and $\boldsymbol\xi =(\xi _1,\ldots,\xi _n)$.
The matrix $\mathfrak S(\bs\xi)$,  is in fact the toroidal/periodic symbol (cf. \cite[Section 4.1.1]{RT-book2010}) of the anisotropic Stokes system \eqref{eq:mik16}-\eqref{eq:mik17}.

\begin{lemma}
\label{ADN-system-ext}
Let $n\ge 2$ and condition \eqref{eq:mik2} hold. Then  the matrix $\mathfrak S(\bs\xi)$ is non-singular for any $\bs\xi\in \dot\Z^n$ and hence the solution of the algebraic system \eqref{eq:mik20}-\eqref{eq:mik21} can be represented in terms of the inverse matrix $\mathfrak S^{-1}(\bs\xi)$ as 
\begin{align}\label{eq:mik22}
\left(\widehat{\mathbf u}(\bs\xi) \atop \hat p(\bs\xi)\right)=\mathfrak S^{-1}(\bs\xi)\left(\widehat{\mathbf f}(\bs\xi) \atop \hat g(\bs\xi)\right) \quad\forall\,\bs\xi\in \dot\Z^n.
\end{align}
Moreover,  the following estimates hold,
\begin{align}
\label{eq:mik23}
&|\widehat{\mathbf u}(\bs\xi)|\le \widehat C_{uf}\frac{|\widehat{\mathbf f}(\bs\xi)|}{|2\pi\bs\xi|^2}
+\widehat C_{ug}\frac{|\hat{g}(\bs\xi)|}{2\pi|\bs\xi|},\\
\label{eq:mik24}
&
|\hat p(\bs\xi)|\le \widehat C_{pf}\frac{|\widehat{\mathbf f}(\bs\xi)|}{2\pi|\bs\xi|}
+\widehat C_{pg}|\hat{g}(\bs\xi)|\quad\forall\,\bs\xi\in \dot\Z^n,
\end{align}
where
$$
 \widehat C_{uf}=2C_{\mathbb A},\ 
\widehat C_{ug}=\widehat C_{pf}=1+2C_{\mathbb A}\|\mathbb A\|,\ 
\widehat C_{pg}=\|\mathbb A\|(1+2C_{\mathbb A}\|\mathbb A\|).
$$
\end{lemma}
\begin{proof}
Introducing the new variables $\widehat{\mathbf w}=2\pi\widehat{\mathbf u}$, $\hat q=i\hat p$, the algebraic system \eqref{E4.7} can be represented in the equivalent form
\begin{align}\label{E4.7e}
\widetilde{\mathfrak S}(\bs\xi)\left(\widehat{\mathbf w}(\bs\xi) \atop \hat q(\bs\xi)\right)=\left(\widehat{\mathbf f}(\bs\xi)/(2\pi) \atop 
-i\hat g(\bs\xi)\right) \quad\forall\,\bs\xi\in \dot\Z^n,
\end{align}
where $\widetilde{\mathfrak S}(\bs\xi)$ is a real matrix with the entries
\begin{align}
\label{ADN-elliptic-1}
\widetilde{\mathfrak S}_{\ell j}(\boldsymbol\xi )=
\begin{cases}
\xi _\alpha a_{\ell j}^{\alpha \beta }\xi _\beta,  &  \ell, j=1,\ldots ,n;\\
\xi _\ell, & \ell= 1,\ldots ,n,\ j=n+1;\\
\xi _j,  & \ell=n+1,\ j=1,\ldots ,n;\\
0,& \ell =j=n+1.
\end{cases}
\end{align}
In order to show that
$\widetilde{\mathfrak S}_{\ell j}(\boldsymbol\xi )$ is non-singular for any $\boldsymbol\xi \in \dot\Z^n$, we
use Theorem \ref{B-B}.
To this end, for a fixed $\boldsymbol\xi \in \dot\Z^n$, we consider the bilinear forms $a_0:{\mathbb R}^n\times {\mathbb R}^n\to {\mathbb R}$ and $b_0:{\mathbb R}^n\times {\mathbb R}\to {\mathbb R}$,
\begin{align}
\label{a-0}
&a_0(\widehat{\mathbf w},\widehat{\mathbf v}):=\widehat{w}_\ell \xi_\alpha a_{\ell j}^{\alpha \beta }\xi _\beta \hat{v}_j\quad \ \forall \, \widehat{\mathbf w},\widehat{\mathbf v}\in {\mathbb R}^n\,,\\
\label{b-0}
&b_0(\widehat{\mathbf v},\hat{q}):=\xi_j\hat{v}_j\hat{q}\hspace{5.5em} \forall \, \widehat{\mathbf v}\in {\mathbb R}^n,\, \hat{q}\in {\mathbb R},
\end{align}
as well as the closed subspace $V_{\boldsymbol\xi}$ of ${\mathbb R}^n$ given by
\begin{align}
\label{V-0}
V_{\boldsymbol\xi}:=\left\{\widehat{\mathbf v}\in {\mathbb R}^n:b_0(\widehat{\mathbf v},\hat{q})=0,\ \forall \, \hat{q}\in {\mathbb R}\right\}=\left\{\widehat{\mathbf v}\in {\mathbb R}^n:\xi_j\hat{v}_j=0\right\}.
\end{align}
It is immediate that these bilinear forms are bounded, as they satisfy the estimates:
\begin{align*}
|a_0(\widehat{\mathbf w},\widehat{\mathbf v})|\leq \|{\mathbb A}\|
|\boldsymbol{\xi }|^2|\widehat{\mathbf w}|\, |\widehat{\mathbf v}|,\ \ |b_0(\widehat{\mathbf v},\hat{q})|\leq |\boldsymbol{\xi }|\, |\widehat{\mathbf v}|\, |\hat{q}| \quad \forall \, \widehat{\mathbf w},\widehat{\mathbf v}\in {\mathbb R}^n,\ \forall \, \hat{q}\in {\mathbb R}.
\end{align*}

The symmetry conditions \eqref{eq:mik1} allow us to write the bilinear form $a_0$ as
\begin{align}
\label{a-0-1}
a_0(\widehat{\mathbf w},\widehat{\mathbf v})
=a_{\ell j}^{\alpha \beta } (\widehat{\mathbf w}\otimes \boldsymbol\xi )_{\ell \alpha }^s(\widehat{\mathbf v}\otimes \boldsymbol\xi )^s_{\beta j}\,,
\end{align}
where $(\widehat{\mathbf w}\otimes \boldsymbol\xi )^s$ is the symmetric part of the matrix $\widehat{\mathbf w}\otimes \boldsymbol\xi$, i.e.,
\begin{align}
(\widehat{\mathbf w}\otimes \boldsymbol\xi )^s_{\ell \alpha}
:=\frac{1}{2}\left(\widehat{w}_\ell \xi_\alpha +\widehat{w}_\alpha \xi_\ell\right),\ \ \ell ,\alpha =1,\ldots ,n\,.
\end{align}
Due to \eqref{a-0-1} and the ellipticity condition \eqref{eq:mik2} we obtain that $a_0$ satisfies the estimate
\begin{align}
a_0(\widehat{\mathbf v},\widehat{\mathbf v})\geq c_{{\mathbb A}}^{-1}|(\widehat{\mathbf v}\otimes {\boldsymbol\xi })^s|^2
=\frac{1}{2}c_{{\mathbb A}}^{-1}|\widehat{\mathbf v}|^2|{\boldsymbol\xi }|^2
\quad  \forall \, \widehat{\mathbf v}\in {\mathbb R}^n \mbox{ such that } \widehat{\mathbf v}\cdot \boldsymbol\xi =0\,,
\end{align}
where $\widehat{\mathbf v}\cdot \boldsymbol\xi =\sum _{\ell =1}^n(\widehat{\mathbf v}\otimes \boldsymbol\xi )_{\ell \ell }^s$ is the trace of the symmetric matrix $(\widehat{\mathbf v}\otimes \boldsymbol\xi )^s$. Therefore, the bounded bilinear form $a_0:V_{\boldsymbol \xi}\times V_{\boldsymbol \xi}\to {\mathbb R}$ is coercive when $\boldsymbol\xi \not= \bf 0$.

In addition, an elementary computation shows that
\begin{align}
\label{inf-sup-b0}
\inf_{\hat{q}\in {\mathbb R}\setminus \{0\}}\sup _{\widehat{\mathbf v}\in {\mathbb R}^n\setminus \{{\bf 0}\}}\frac{b_0(\widehat{\mathbf v},\hat{q})}{|\widehat{\mathbf v}|\, |\hat{q}|}\geq |\boldsymbol\xi |\,,
\end{align}
and accordingly that the bilinear form $b_0$ satisfies the inf-sup condition with the inf-sup constant $|\boldsymbol\xi |$.

By applying Theorem \ref{B-B}, we conclude that the modified symbol matrix $\widetilde{\mathfrak S} (\boldsymbol\xi )$ given by \eqref{ADN-elliptic-1} is invertible for any $\boldsymbol\xi \in \dot\Z^n$, and hence that the symbol matrix ${\mathfrak S} (\boldsymbol\xi )$ given by \eqref{ADN-elliptic} has the same property and the formula \eqref{eq:mik22} is well defined. 
Moreover, estimates \eqref{mixed-Cu} and \eqref{mixed-Cp} are applicable to the real and imaginary parts of system \eqref{E4.7e} and after combining them, we get
\begin{align*}
&|\widehat{\mathbf w}(\bs\xi)|\le \widehat C_{uf}\frac{|\widehat{\mathbf f}(\bs\xi)|}{2\pi|\bs\xi|^2}
+\widehat C_{ug}\frac{|\hat{g}(\bs\xi)|}{|\bs\xi|},\\
&
|\hat q(\bs\xi)|\le \widehat C_{pf}\frac{|\widehat{\mathbf f}(\bs\xi)|}{2\pi|\bs\xi|}
+\widehat C_{pg}|\hat{g}(\bs\xi)|\quad\forall\,\bs\xi\in \dot\Z^n.
\end{align*}
Recalling that $\widehat{\mathbf w}=2\pi\widehat{\mathbf u}$ and $\hat q=i\hat p$, we arrive at estimates \eqref{eq:mik23} and \eqref{eq:mik24}.
\end{proof}

Note that the proof of Lemma \ref{ADN-system-ext} is essentially similar to the proof  of \cite[Lemma 15]{KMW-LP2021} given for the non-periodic case, where the ellipticity of the anisotropic Stokes system in the sense of Agmon–Douglis– Nirenberg was given,  
but augments  it with estimates \eqref{eq:mik23}, \eqref{eq:mik24}.

\begin{remark}\label{rem:mik1}
For the isotropic case \eqref{eq:mik7}, 
due to \eqref{eq:mik8},
system \eqref{eq:mik20}-\eqref{eq:mik21} reduces to
\begin{align}
\label{eq:mik25}
&4\pi^2\left[(\lambda+\mu) \bs\xi(\bs\xi\cdot \widehat{\mathbf u}(\bs\xi))+\mu|\bs\xi|^2\widehat{\mathbf u}(\bs\xi)\right]
+2\pi i\bs\xi\hat p(\bs\xi)=\widehat{\mathbf f}(\bs\xi), \quad\forall\,\bs\xi\in \dot\Z^n,\\
\label{eq:mik26}
&2\pi i\bs\xi\cdot \widehat{\mathbf u}(\bs\xi)=\hat g(\bs\xi) 
\quad\forall\,\bs\xi\in \dot\Z^n.
\end{align}
Taking scalar product of equation \eqref{eq:mik25} with $\bs\xi$ and employing \eqref{eq:mik26}, we obtain
\begin{align}
\label{eq:mik27}
&\hat p(\bs\xi)=\frac{\bs\xi\cdot\widehat{\mathbf f}(\bs\xi)}{2\pi i|\bs\xi|^2} +(\lambda+2\mu) {\hat g(\bs\xi)}, \quad\forall\,\bs\xi\in \dot\Z^n,
\end{align}
and substituting this back to \eqref{eq:mik25}, we get
\begin{align}
\label{eq:mik28}
& 
\widehat{\mathbf u}(\bs\xi)
=\frac{1}{4\pi^2\mu|\bs\xi|^2}\left[\widehat{\mathbf f}(\bs\xi) 
-\bs\xi\frac{\bs\xi\cdot\widehat{\mathbf f}(\bs\xi)}{|\bs\xi|^2}\right]
+ \bs\xi\frac{\hat g(\bs\xi)}{2\pi i|\bs\xi|^2},
\quad\forall\,\bs\xi\in \dot\Z^n
\end{align}
(cf. \cite[Section 2.2]{Temam1995} for the case $s=1$, $g=0$, $\lambda=0$, and $\mu=1$).
Expressions  \eqref{eq:mik28}, \eqref{eq:mik27}  evidently satisfy estimates \eqref{eq:mik23}, \eqref{eq:mik24}.
\hfill$\square$
\end{remark}

The anisotropic Stokes system \eqref{eq:mik16}-\eqref{eq:mik17} can be re-written as  
\begin{align*}
&S\left({\mathbf u}\atop p\right)=\left({\mathbf f} \atop g\right),
\end{align*}
where
\begin{align*}
&S\left({\mathbf u}\atop p\right):=\left(-\bs{\mathfrak L}{\mathbf u}+\nabla p \atop {\rm{div}}\, {\mathbf u}\right),
\end{align*}
and for any $s\in\R$,
\begin{align}\label{eq:mik29}
{S}:\dot{\mathbf H}_\#^s\times \dot H_\#^{s-1} \to \dot{\mathbf H}_\#^{s-2}\times \dot H_\#^{s-1} 
\end{align}
is a linear continuous operator.

Now we are in the position to prove the following assertion.
\begin{theorem}\label{th:mik1}
Let $n\ge 2$ and condition \eqref{eq:mik2} hold.

(i)
For any
$({\mathbf f},g)\in\dot{\mathbf H}_\#^{s-2}\times \dot H_\#^{s-1}$, $s\in\R$,
the anisotropic Stokes system \eqref{eq:mik16}-\eqref{eq:mik17} in torus $\T$ has a unique solution
$({\mathbf u},p )\in \dot{\mathbf H}_\#^s\times \dot H_\#^{s-1}$, where
\begin{align}\label{eq:mik30}
{\mathbf u}(\mathbf x)=\sum_{\bs\xi\in\dot\Z^n}e^{2\pi i x\cdot\bs\xi}\widehat {\mathbf u}(\bs\xi),\quad
p(\mathbf x)=\sum_{\bs\xi\in\dot\Z^n}e^{2\pi i x\cdot\bs\xi}\hat p(\bs\xi)
\end{align}
with $\widehat {\mathbf u}(\bs\xi)$ and $\hat p(\bs\xi)$ given by \eqref{eq:mik22}. 
In addition, 
\begin{align}
\label{eq:mik31u}
&\|{\mathbf u}\|_{\dot{\mathbf H}_\#^{s}}
\leq C_{uf}\|\mathbf f \|_{\dot{\mathbf H}_\#^{s-2}}+C_{ug}\|g\|_{\dot H_\#^{s-1} },\\
\label{eq:mik31p}
&\|p \|_{\dot H_\#^{s-1} }
\leq C_{pf}\|\mathbf f \|_{\dot{\mathbf H}_\#^{s-2}}+C_{pg}\|g\|_{\dot H_\#^{s-1} },
\end{align}
where
\begin{align*}
C_{uf}=\frac{1}{\pi^2}C_{\mathbb A},\ 
C_{ug}=C_{pf}=\frac{1}{\sqrt{2}\pi}(1+2C_{\mathbb A}\|\mathbb A\|),\ 
C_{pg}=\|\mathbb A\|(1+2C_{\mathbb A}\|\mathbb A\|),
\end{align*}
and operator \eqref{eq:mik29} is an isomorphism.

(ii) Moreover, if $({\mathbf f},g)\in\dot{\mathbf C}^\infty_{\#}\times \dot {\mathcal C}^\infty_\#$
then $({\mathbf u},p)\in\dot{\mathbf C}^\infty_{\#}\times \dot {\mathcal C}^\infty_\#$.
\end{theorem}
\begin{proof}
(i) Expressions \eqref{eq:mik22} supplemented by the relations $\widehat{\mathbf u}(\mathbf 0)=\mathbf 0$, $\hat p(\mathbf 0)=0$,
following from the inclusions $({\mathbf u},p )\in \dot{\mathbf H}_\#^s\times \dot H_\#^{s-1}$, imply the uniqueness. 
From estimates \eqref{eq:mik23} 
and \eqref{eq:mik24} 
we obtain the estimates
\begin{align*}
\| {\mathbf u}\|_{\dot{\mathbf H}_\#^s}&=\left(\sum_{\bs\xi\in\dot\Z^n}\rho(\bs\xi)^{2s}|\widehat {\mathbf u}(\bs\xi)|^2\right)^{1/2}\nonumber\\
&\le \frac{\widehat C_{uf}}{4\pi^2}\left(\sum_{\bs\xi\in\dot\Z^n}\rho(\bs\xi)^{2s}
\frac{|\widehat{\mathbf f}(\bs\xi)|^2}{|\bs\xi|^4}\right)^{1/2}
+\frac{\widehat C_{ug}}{2\pi}\left(\sum_{\bs\xi\in\dot\Z^n} \rho(\bs\xi)^{2s} \frac{|\hat{g}(\bs\xi)|^2}{|\bs\xi|^2}\right)^{1/2}\nonumber\\
&= \frac{\widehat C_{uf}}{4\pi^2}\left(\sum_{\bs\xi\in\dot\Z^n}\rho(\bs\xi)^{2(s-2)}|\widehat{\mathbf f}(\bs\xi)|^2
\frac{\rho(\bs\xi)^4}{|\bs\xi|^4}\right)^{1/2}\nonumber\\
&\hspace{10.5em}+\frac{\widehat C_{ug}}{2\pi}\left(\sum_{\bs\xi\in\dot\Z^n}\rho(\bs\xi)^{2(s-1)}|\hat{g}(\bs\xi)|^2
\frac{\rho(\bs\xi)^2}{|\bs\xi|^2}\right)^{1/2}\nonumber\\
&\le \frac{\widehat C_{uf}}{2\pi^2}\|\mathbf f\|_{\dot{\mathbf H}_\#^{s-2}}
+\frac{\widehat C_{ug}}{\sqrt{2}\pi}\|g\|_{\dot H_\#^{s-1} },
\\
\|p \|_{\dot H_\#^{s-1}}&=\left(\sum_{\bs\xi\in\dot\Z^n}\rho(\bs\xi)^{2s-2}|\hat p(\bs\xi)|^2\right)^{1/2}\nonumber\\
&\le \frac{\widehat C_{pf}}{2\pi}\left(\sum_{\bs\xi\in\dot\Z^n}\rho(\bs\xi)^{2s-2}
\frac{|\widehat{\mathbf f}(\bs\xi)|^2}{|\bs\xi|^2}\right)^{1/2}
+{\widehat C_{pg}}\left(\sum_{\bs\xi\in\dot\Z^n} \rho(\bs\xi)^{2s-2} {|\hat{g}(\bs\xi)|^2}\right)^{1/2}\nonumber\\
&= \frac{\widehat C_{pf}}{2\pi}\left(\sum_{\bs\xi\in\dot\Z^n}\rho(\bs\xi)^{2(s-2)}|\widehat{\mathbf f}(\bs\xi)|^2
\frac{\rho(\bs\xi)^2}{|\bs\xi|^2}\right)^{1/2}\nonumber\\
&\hspace{15em}+{\widehat C_{pg}}\left(\sum_{\bs\xi\in\dot\Z^n}\rho(\bs\xi)^{2(s-1)}|\hat{g}(\bs\xi)|^2\right)^{1/2}\nonumber\\
&\le \frac{\widehat C_{pf}}{\sqrt{2}\pi}\|\mathbf f\|_{\dot{\mathbf H}_\#^{s-2}}
+{\widehat C_{pg}}\|g\|_{\dot H_\#^{s-1} }.
\end{align*}
These estimates imply \eqref{eq:mik31u}-\eqref{eq:mik31p} and hence inclusions in the corresponding spaces and that operator \eqref{eq:mik29} is an isomorphism.

(ii) The inclusion $({\mathbf f},g)\in\dot{\mathbf C}^\infty_{\#}\times \dot {\mathcal C}^\infty_\#$ implies that $({\mathbf f},g)\in\dot{\mathbf H}_\#^{s-2}\times \dot H_\#^{s-1}$ for any $s\in\R$.
Then  by item (i) $({\mathbf u},p )\in \dot{\mathbf H}_\#^s\times \dot H_\#^{s-1} $ for any $s\in\R$  and hence
$({\mathbf u},p)\in\dot{\mathbf C}^\infty_{\#}\times \dot {\mathcal C}^\infty_\#$.
\mg
\end{proof}

If $g=0$ in \eqref{eq:mik17}, we can re-formulate the Stokes system \eqref{eq:mik16}-\eqref{eq:mik17} as the vector equation
\begin{align}
\label{eq:mik32}
&-\bs{\mathfrak L}{\mathbf u}+\nabla p=\mathbf{f}
\end{align}
for the unknowns $({\mathbf u},p )\in \dot{\mathbf H}_{\#\sigma}^{s}\times \dot H_\#^{s-1}$ and the given data
${\mathbf f}\in\dot{\mathbf H}_\#^{s-2}$, $s\in\R$.
Then Theorem \ref{th:mik1} implies the following assertion.
\begin{corollary}\label{cor:mik1}
Let $n\ge 2$ and condition \eqref{eq:mik2} hold.

(i)
For any
${\mathbf f}\in\dot{\mathbf H}_\#^{s-2}$, $s\in\R$,
the anisotropic Stokes equation \eqref{eq:mik32} in torus $\T$ has a unique incompressible  solution
$({\mathbf u},p )\in \dot{\mathbf H}_{\#\sigma}^{s}\times \dot H_\#^{s-1} $, 
with $\widehat {\mathbf u}(\bs\xi)$ and $\hat p(\bs\xi)$ given by \eqref{eq:mik22}, \eqref{eq:mik30} (and particularly by \eqref{eq:mik28}, \eqref{eq:mik27}, \eqref{eq:mik30} for the isotropic case \eqref{eq:mik7})  with $g=0$.
In addition, 
\begin{align}
\label{eq:mik31u0}
&\|{\mathbf u}\|_{\dot{\mathbf H}_\#^{s}}
\leq C_{uf}\|\mathbf f \|_{\dot{\mathbf H}_\#^{s-2}},\\
\label{eq:mik31p0}
&\|p \|_{\dot H_\#^{s-1} }
\leq C_{pf}\|\mathbf f \|_{\dot{\mathbf H}_\#^{s-2}},
\end{align}
where
\begin{align*}
C_{uf}=\frac{1}{\pi^2}C_{\mathbb A},\ 
C_{pf}=\frac{1}{\sqrt{2}\pi}(1+2C_{\mathbb A}\|\mathbb A\|),
\end{align*}
and the operator 
$$\bs{\mathcal L}: \dot{\mathbf H}_{\#\sigma}^{s}\times \dot H_\#^{s-1} \to \dot{\mathbf H}_\#^{s-2},$$
where
\begin{align}
\label{eq:mik4}
\bs{\mathcal L}({\mathbf u},p ):=\boldsymbol{\mathfrak L}{\mathbf u}-\nabla p,
\end{align}
 is an isomorphism.

(ii) Moreover, if ${\mathbf f}\in\dot{\mathbf C}^\infty_{\#}$
then $({\mathbf u},p)\in\dot{\mathbf C}^\infty_{\#\sigma}\times \dot {\mathcal C}^\infty_\#$.
\end{corollary}

\section{Stationary anisotropic periodic Oseen system}\label{sec:mik4t}


For the unknowns $({\mathbf u},p )\in \dot{\mathbf H}_\#^s\times \dot H_\#^{s-1}$, 
let us consider the Oseen system
\begin{align}
\label{Oseen-problem}
-\bs{\mathfrak L}{\mathbf u}+\nabla p+({\mathbf U}\cdot \nabla ){\mathbf u}&=\mathbf{f}, 
\\
\label{Oseen-problem-div}
{\rm{div}}\, {\bf u}&=g
\end{align}
with given data
$({\mathbf f},g)\in\dot{\mathbf H}_\#^{s-2}\times \dot H_\#^{s-1}$, $s\ge 1$, and a given function ${\mathbf U}$.

The anisotropic Oseen system \eqref{Oseen-problem}-\eqref{Oseen-problem-div} can be re-written as  
\begin{align*}
&S_U\left({\mathbf u}\atop p\right)=\left({\mathbf f} \atop g\right),
\end{align*}
where
\begin{align}\label{E5.2}
&S_U\left({\mathbf u}\atop p\right):
=\left(-\bs{\mathfrak L}{\mathbf u}+\nabla p+({\mathbf U}\cdot \nabla ){\mathbf u}\atop {\rm{div}}\, {\mathbf u}\right)
=S\left({\mathbf u}\atop p\right) + \left(({\mathbf U}\cdot \nabla ){\mathbf u} \atop 0\right).
\end{align}

\subsection{Weak solution to the stationary \mg periodic anisotropic Oseen system}

For a fixed ${\mathbf U}$, 
and the bilinear forms 
\begin{align}
\label{a-v}
&a_{\T;\mathbf U}({\bf u},{\bf v})
:=
\left\langle a_{ij}^{\alpha \beta }E_{j\beta }({\bf u}),E_{i\alpha }({\bf v})\right\rangle _{\T}
+\left\langle ({\mathbf U}\cdot \nabla ){\mathbf u},{\bf v}\right\rangle _{\T}\,,\
\forall \ {\bf u}, {\bf v}\in \dot{\mathbf H}_\#^1\,,\\
\label{b-v}
&b_{\T}({\bf v},q):=-\langle {\rm{div}}\, {\bf v},q\rangle _{\T}\,,\quad \forall \ {\bf v}\in \dot{\mathbf H}_\#^1\,, \ 
\forall \, q\in \dot L_{2\#}\,,
\end{align}
let us consider the following mixed variational problem:\\
{\em Find $({\bf u},p )\in {\dot{\mathbf H}_\#^1}\times \dot L_{2\#}$ such that for given ${\bf f} \in \dot{\mathbf H}_\#^{-1}$ and $g \in \dot L_{2\#}$,
}
\begin{align}
\label{transmission-S-variational-dl-3-equiv-0-2}
\left\{\begin{array}{ll}
a_{\T;\mathbf U}({\bf u},{\bf v})+b_{\T}({\bf v},p )
=-\langle{\bf f},{\bf v}\rangle_{\T}\quad \forall \, {\bf v}\in \dot{\mathbf H}_\#^1,\\
b_{\T}({\bf u},q)=-\langle g ,q\rangle _{\T}\quad \forall \, q\in \dot L_{2\#}.
\end{array}
\right.
\end{align}

Let us note that the subspace $\dot{\mathbf H}_{\#\sigma}^1$ of $\dot{\mathbf H}_\#^1$, see \eqref{E3.8}, has also the characterization
\begin{align}
\label{E3.4}
\dot{\mathbf H}_{\#\sigma}^1=\left\{{\bf w}\in \dot{\mathbf H}_\#^1: b_{\T}({\bf w},q)=0\quad \forall \, q\in \dot L_{2\#} \right\}.
\end{align}

Let us now prove the well-posedness result for problem \eqref{transmission-S-variational-dl-3-equiv-0-2} (cf.  Lemma 3.1 in \cite{KMW2020}, and Lemma 5  in \cite{KMW-LP2021}, where similar results were proved for the Stokes system in non-periodic anisotropic settings).
\begin{theorem}
\label{Oseen-problemTh}
Let $n\ge 2$, condition \eqref{eq:mik2} hold, and ${\mathbf U}\in {\mathbf L}_{\theta\#\sigma}$, where $\theta\in(2,\infty)$ if $n=2$, while  $\theta\in[n,\infty)$ if $n\ge 3$.

(i) Then for all given data ${\bf f} \in \dot{\mathbf H}_\#^{-1}$ and $g \in \dot L_{2\#}$, the variational problem \eqref{transmission-S-variational-dl-3-equiv-0-2}
has a unique solution $({\bf u},p )\in {\dot{\mathbf H}_\#^1}\times \dot L_{2\#}$ and 
\begin{align}
\label{estimate-1-wp-S-2u}
&\|{\bf u}\|_{\dot{\mathbf H}_\#^1}
\leq C_{uf}\|{\bf f} \|_{\dot{\mathbf H}_\#^{-1}}+C_{ug;U}\|g \|_{\dot L_{2\#}},\\
\label{estimate-1-wp-S-2p}
&\|p \|_{\dot L_{2\#}}
\leq C_{pf;U}\|{\bf f} \|_{\dot{\mathbf H}_\#^{-1}}+C_{pg;U}\|g \|_{\dot L_{2\#}},
\end{align}
where 
\begin{align}\label{Cuf}
C_{uf}=\pi^{-2}C_{\mathbb A},
\end{align}
while the constants
$C_{ug;U}=C_{pf;U}$ and 
$C_{pg;U}$ depend only on $C_{\mathbb A}$, $\|\mathbb A\|$, $n$, and $\mathbf U$.

(ii) Moreover, the anisotropic Oseen system \eqref{Oseen-problem}-\eqref{Oseen-problem-div}
is well-posed in $\dot{\mathbf H}_\#^1\times  \dot L_{2\#}$ and its unique solution $({\bf u},p)\in {\dot{\mathbf H}_\#^1}\times \dot L_{2\#}$ is provided by  the solution of the variational problem from item (i).

(iii) The operator 
\begin{align}\label{eq:mik29U1}
{S}_U:\dot{\mathbf H}_\#^1\times \dot L_{2\#} \to \dot{\mathbf H}_\#^{-1}\times \dot L_{2\#}, 
\end{align}
where $S_U$ is defined by \eqref{E5.2}, is an isomorphism.
\end{theorem}
\begin{proof}
(i) We intend to use Theorem \ref{B-B}, which requires boundedness  of the bilinear form 
$a_{\T;\mathbf U}(\cdot ,\cdot ):\dot{\mathbf H}_\#^1\times \dot{\mathbf H}_\#^1\to\mathbb R$ and coercivity of the bilinear form
$a_{\T;\mathbf U}(\cdot ,\cdot ):\dot{\mathbf H}_{\#\sigma}^1\times \dot{\mathbf H}_{\#\sigma}^1\to {\mathbb R}$.

Indeed, by \eqref{TensNorm} and \eqref{eq:mik14},
\begin{multline}
\left|\left\langle a_{ij}^{\alpha \beta }E_{j\beta }({\bf u}),E_{i\alpha }({\bf v})\right\rangle _{\T}\right|
\le n^4 \|{\mathbb A}\| \|{\mathbb E}({\bf u})\|_{L_{2\#}^{n\times n}} \|{\mathbb E}({\bf v})\|_{L_{2\#}^{n\times n}}\\
\le n^4 \|{\mathbb A}\| \|\nabla{\bf u}\|_{L_{2\#}^{n\times n}} \|\nabla{\bf v}\|_{L_{2\#}^{n\times n}}
 \le 4\pi^2n^4 \|{\mathbb A}\| \| \mathbf u\|_{\dot{\mathbf H}_\#^1} \| \mathbf v\|_{\dot{\mathbf H}_\#^1}.
\end{multline}
On the other hand, by \eqref{eq:mik37a2a} with $\mathbf U$ for $\mathbf v_1$ and ${\mathbf u}$ for $\mathbf v_2$
\begin{align}
\left|\left\langle ({\mathbf U}\cdot \nabla ){\mathbf u},{\bf v}\right\rangle _{\T}\right|
&\le \|({\mathbf U}\cdot \nabla ){\mathbf u}\|_{{\mathbf H}^{-1}_{\#}} \|{\bf v}\|_{{\mathbf H}^{1}_{\#}}
\nonumber\\
&\le {\cn 2\pi}{\mg C_{2\theta/(\theta-2)\#}} \|{\mathbf U}\|_{{\mathbf L}_{\gr\theta\#}} \|{\mathbf u}\|_{{\mathbf H}^1_{\#}}\|{\bf v}\|_{{\mathbf H}^{1}_{\#}}.
\end{align}
Hence
\begin{align}
|a_{\T;\mathbf U}({\bf u},{\bf v})|\le (4\pi^2n^4 \|{\mathbb A}\| 
+{\cn 2\pi}{\mg C_{2\theta/(\theta-2)\#}} \|{\mathbf U}\|_{{\mathbf L}_{\gr\theta\#}})
\| \mathbf u\|_{\dot{\mathbf H}_\#^1} \| \mathbf v\|_{\dot{\mathbf H}_\#^1},
\end{align}
which proves boundedness of the bilinear form 
$a_{\T;\mathbf U}(\cdot ,\cdot ):\dot{\mathbf H}_\#^1\times \dot{\mathbf H}_\#^1\to\mathbb R$.

The first Korn inequality \eqref{eq:mik15}, the relation $\sum_{i=1}^nE_{ii}({\bf w})=\div {\bf w}=0$ for  ${\bf w}\in 
{\dot{\mathbf H}_{\#\sigma}^1}$,
the ellipticity condition \eqref{eq:mik2}, and  equivalence of the norm $\|\nabla (\cdot )\|_{L_{2\#}^{n\times n}}$
to the norm $\|\cdot \|_{\dot{\mathbf H}_\#^1}$ in $\dot{\mathbf H}_\#^1$, see \eqref{eq:mik14},
imply that 
\begin{align}
\label{a-1-v2-S-}
\left\langle a_{ij}^{\alpha \beta }E_{j\beta }({\bf w}),E_{i\alpha }({\bf w})\right\rangle _{\T}
&\geq C_{\mathbb A}^{-1}\|{\mathbb E}({\bf w})\|_{L_{2\#}^{n\times n}}^2
\geq\frac{1}{2}C_{\mathbb A}^{-1}\|\nabla {\bf w}\|_{L_{2\#}^{n\times n}}^2
\nonumber\\
&\geq \pi^2C_{\mathbb A}^{-1}\|{\bf w}\|_{\dot{\mathbf H}_\#^1}^2\quad
\forall \, {\bf w}\in \dot{\mathbf H}_{\#\sigma}^1.
\end{align}
On the other hand, since ${\mathbf U}\in {\mathbf L}_{\theta\#\sigma}$,  Section~\ref{S7.3.3}(iii) implies that
\begin{align*}
\left\langle ({\mathbf U}\cdot \nabla ){\mathbf w},{\bf w}\right\rangle _{\T}&=0 \quad
\forall \, {\bf w}\in {\mathbf H}_{\#}^1.
\end{align*}
Hence
\begin{align}\label{a-1-v2-S}
a_{\T;\mathbf U}({\bf w},{\bf w})\geq\pi^2C_{\mathbb A}^{-1}\|{\bf w}\|_{\dot{\mathbf H}_\#^1}^2\quad
\forall \, {\bf w}\in \dot{\mathbf H}_{\#\sigma}^1.
\end{align}
Inequality \eqref{a-1-v2-S} shows that the bilinear form
$a_{\T;\mathbf U}(\cdot ,\cdot ):\dot{\mathbf H}_{\#\sigma}^1\times \dot{\mathbf H}_{\#\sigma}^1\to {\mathbb R}$ is coercive.

The boundedness of the divergence operator
${\rm{div}}:\dot{\mathbf H}_\#^1\to \dot L_{2\#}$
implies that the bilinear form $b_\T:\dot{\mathbf H}_\#^1\times \dot L_{2\#}\to {\mathbb R}$ is bounded as well.
One can also check that for any $g\in\dot L_{2\#}$, the divergence PDE \eqref{Oseen-problem-div} has a unique solution in $\dot{\mathbf H}_\#^1/\dot{\mathbf H}_{\#\sigma}^1$. 
The PDE is equivalent to
 the algebraic equation \eqref{eq:mik21} for the Fourier coefficients,  and the solution is presented  as
\begin{align*}
\widehat{\mathbf u}(\bs\xi)&=\frac{\bs\xi}{2\pi i|\bs\xi|^2}\hat g(\bs\xi) \quad\ \forall\,\bs\xi\in \dot\Z^n.
\end{align*}
Moreover, then the divergence operator
\begin{align*}
-{\rm{div}}:\dot{\mathbf H}_\#^1/\dot{\mathbf H}_{\#\sigma}^1\to \dot L_{2\#},
\end{align*}
is an isomorphism,
cf. also, e.g., \cite[Lemmas 7-9 in p. 30]{Tartar1978},
\cite[Corollary 2.4 and Theorem 2.3 in Chapter 1]{Girault-Raviart1986},
\cite[Proposition 1.2(i) and Remark 1.4 in Chapter 1]{Temam2001}, 
\cite[Theorem 3.1]{ACM2015} and 
\cite[Theorem 3.1]{KMW-DCDS2021} 
for some non-periodic settings.
The isomorphism implies that
there exists a constant $c_0>0$ such that for any $q\in \dot L_{2\#}$ there exists ${\bf v}_q\in \dot{\mathbf H}_\#^1$ satisfying the equation $-{\rm{div}}\, {\bf v}_q=q$ and the inequality
$\|\mathbf v\|_{\dot{\mathbf H}_\#^1}\leq {c_0}\|q\|_{L_2\#}$. Therefore, the following inequality holds for such $\bf v$,
\begin{align*}
b_{\T}({\bf v}_q,q)=-\left\langle {\rm{div}}\, {\bf v}_q,q\right\rangle _{\T}
=\langle q,q\rangle _{\T}=\|q\|_{\dot L_{2\#}}^2
\geq c_0^{-1}\|{\bf v}_q\|_{\dot{\mathbf H}_\#^1}\|q\|_{\dot L_{2\#}}.
\end{align*}
This, in turn, implies that the bounded bilinear form $b_{\T}:\dot{\mathbf H}_\#^1\times \dot L_{2\#}\to {\mathbb R}$ satisfies the inf-sup condition
\begin{align*}
\inf _{q\in \dot L_{2\#}\setminus \{0\}}\sup _{{\bf w}\in \dot{\mathbf H}_\#^1\setminus \{\bf 0\}}\frac{b_{\T}({\bf w},q)}{\|{\bf w}\|_{\dot{\mathbf H}_\#^1}\|q\|_{\dot L_{2\#}}}
\geq \inf _{q\in \dot L_{2\#}\setminus \{0\}}\frac{b_{\T}({\bf v}_q,q)}{\|{\bf v}_q\|_{\dot{\mathbf H}_\#^1}\|q\|_{\dot L_{2\#}}}\geq c_0^{-1}.
\end{align*}
Then Theorem \ref{B-B} with $X=\dot{\mathbf H}_\#^1$, ${\mathcal M}=\dot L_{2\#}$, and $V=\dot{\mathbf H}_{\#\sigma}^1$ implies that problem \eqref{transmission-S-variational-dl-3-equiv-0-2} is well-posed, as asserted.

(ii) 
Due to  \eqref{E5.2} and \eqref{eq:mik37c2a}, operator \eqref{eq:mik29U1} is linear and continuous.

The dense embedding of the space $\dot{\mathbf C}^\infty_{\#}$ in ${\dot{\mathbf H}_\#^1}$ shows that system \eqref{Oseen-problem}-\eqref{Oseen-problem-div} has the equivalent mixed variational formulation \eqref{transmission-S-variational-dl-3-equiv-0-2} thus proving item (ii).

(iii)
Estimates \eqref{estimate-1-wp-S-2u}-\eqref{estimate-1-wp-S-2p} imply the existence of a continuous inverse to operator \eqref{eq:mik29U1} 
\hfill \end{proof}

\subsection{\mg Solution regularity for the stationary anisotropic Oseen system}

For simplicity, we will further limit ourself with the case ${\mathbf U}\in {\mathbf C}^\infty_{\#\sigma}$. Regularity of the non-periodic isotropic Oseen problems with less smooth ${\mathbf U}$ were considered, e.g., in \cite{ARB2010}, \cite{ARB2011}, and references therein.

\begin{theorem}\label{th:mik1-Os-inf}
Let $n\ge 2$ and condition \eqref{eq:mik2} hold.
Let ${\mathbf U}\in {\mathbf C}^\infty_{\#\sigma}$.

(i)
For any
$({\mathbf f},g)\in\dot{\mathbf H}_\#^{s-2}\times \dot H_\#^{s-1}$, $s\ge 1$,
the anisotropic Oseen system \eqref{Oseen-problem}-\eqref{Oseen-problem-div} has a unique solution
\begin{align}\label{E5.16incl}
({\mathbf u},p )\in \dot{\mathbf H}_\#^s\times \dot H_\#^{s-1}.
\end{align}
In addition, there exists a constant $C_s=C_s(C_{\mathbb A},\|\mathbb A\|, n,\mathbf U, s)>0$ such that
\begin{align}
\label{eq:mik31Uinf}
\|{\mathbf u}\|_{\dot{\mathbf H}_\#^{s}}+\|p \|_{\dot H_\#^{s-1} }
\leq C_s\left(\|\mathbf f \|_{\dot{\mathbf H}_\#^{s-2}}+\|g\|_{\dot H_\#^{s-1} }\right)
\end{align}
and the operator 
\begin{align}\label{eq:mik29Uinf}
{S}_U: \dot{\mathbf H}_\#^s\times \dot H_\#^{s-1} \to \dot{\mathbf H}_\#^{s-2}\times \dot H_\#^{s-1} , 
\end{align}
 is an isomorphism; $S_U$ is defined by \eqref{E5.2}.

(ii) Moreover, if 
$({\mathbf f},g)\in\dot{\mathbf C}^\infty_{\#}\times \dot {\mathcal C}^\infty_\#$
then $({\mathbf u},p)\in\dot{\mathbf C}^\infty_{\#}\times \dot {\mathcal C}^\infty_\#$.
\end{theorem}
\begin{proof}
Since  ${\mathbf U}\in {\mathbf C}^\infty_{\#\sigma}\subset {\mathbf L}_{\theta\#\sigma}$ for any
$\theta>1$,
Theorem \ref{Oseen-problemTh} implies that system \eqref{Oseen-problem}-\eqref{Oseen-problem-div} has a unique solution 
$({\bf u},p)\in \dot{\mathbf H}_\#^{s_1}\times  \dot{H}_{\#}^{s_1-1}$ with $s_1=1$.
Then $({\mathbf U}\cdot \nabla ){\mathbf u}\in  \dot{\mathbf H}_\#^{s_1-1}$ and
\begin{align}\label{E5.18-}
\|({\mathbf U}\cdot \nabla ){\mathbf u}\|_{\dot{\mathbf H}_\#^{s_1-1}}
\le C_0(\mathbf U)\|{\mathbf u}\|_{\dot{\mathbf H}_\#^{s_1}}
\end{align}
implying, due to estimate \eqref{estimate-1-wp-S-2u} in Theorem \ref{Oseen-problemTh},
\begin{multline}\label{E5.18}
\|({\mathbf U}\cdot \nabla ){\mathbf u}\|_{\dot{\mathbf H}_\#^{s_1-1}}
\le C_0(\mathbf U)\left(C_{uf}\|\mathbf f \|_{\dot{\mathbf H}_\#^{s_1-2}}+C_{ug;U}\|g\|_{\dot H_\#^{s_1-1} }\right)\\
\le C_{1s}(\mathbf U)\left(\|\mathbf f \|_{\dot{\mathbf H}_\#^{s-2}}+\|g\|_{\dot H_\#^{s-1} }\right).
\end{multline}
Hence the couple $({\mathbf u},p)$ satisfies the system
\begin{align}
\label{eq:mik16Uinf}
-\bs{\mathfrak L}{\mathbf u}+\nabla p&=\mathbf{f}^{(1)},\\
\label{eq:mik17Uinf}
{\rm{div}}\, {\mathbf u}&=g
\end{align}
with $\mathbf f^{(1)}:=\mathbf f -({\mathbf U}\cdot \nabla ){\mathbf u}\in \dot{\mathbf H}_\#^{s^{(1)}-2}$, 
where $s^{(1)}=\min\{s, s_1+1\}$.
By Theorem~\ref{th:mik1}(i), the Stokes system \eqref{eq:mik16Uinf}-\eqref{eq:mik17Uinf} has a unique solution in 
$\dot{\mathbf H}_{\#\sigma}^{\tilde s}\times \dot H_\#^{\tilde s-1}$ for any $\tilde s\le s^{(1)}$ and thus 
$({\mathbf u},p)\in \dot{\mathbf H}_{\#\sigma}^{s^{(1)}}\times \dot H_\#^{s^{(1)}-1}$
with the estimate
\begin{align}
\label{eq:mik31Uinf(1)}
\|{\mathbf u}\|_{\dot{\mathbf H}_\#^{s^{(1)}}}+\|p \|_{\dot H_\#^{s^{(1)}-1} }
\leq C^{(1)}\left(\|\mathbf f \|_{\dot{\mathbf H}_\#^{s-2}}+\|g\|_{\dot H_\#^{s-1} }\right)
\end{align}
implied by estimates \eqref{eq:mik31u}-\eqref{eq:mik31p} and \eqref{E5.18}; $C^{(1)}>0$ is a constant depending only on $C_{\mathbb A}$, $\|\mathbb A\|$, $n$, and $\mathbf U$.
If $s^{(1)}=s$,  this proves 
inclusion \eqref{E5.16incl} and estimate \eqref{eq:mik31Uinf}.

Otherwise $s^{(1)}=s_1+1<s$ and we arrange an iterative process by replacing in the previous paragraph $s_1$ with $s^{(1)}$ on each iteration until we arrive at the case $s^{(1)}=\min\{s, s_1+1\}=s$.
Note that in each iteration, $s_1$ increases by 1, which implies that the iteration process will stop after a finite number of iterations.
This proves inclusion \eqref{E5.16incl} and estimate \eqref{eq:mik31Uinf}.

The continuity of operator \eqref{eq:mik29} and estimate \eqref{E5.18-} together with representation \eqref{E5.2}  imply the continuity of operator \eqref{eq:mik29Uinf}.
Along with the existence of a continuous inverse to operator \eqref{eq:mik29Uinf} implied by estimate \eqref{eq:mik31Uinf}, this means
that  operator \eqref{eq:mik29Uinf} is an isomorphism.

 Moreover, if 
$({\mathbf f},g)\in\dot{\mathbf C}^\infty_{\#}\times \dot {\mathcal C}^\infty_\#$, item (i) implies that $({\mathbf u},p)\in \dot{\mathbf H}_{\#\sigma}^{s}\times \dot H_\#^{s-1}$ for arbitrary $s$ thus giving item (ii) of the theorem.
\end{proof}

\section{Stationary anisotropic periodic Navier-Stokes system}\label{sec:mik5}

\subsection{Existence of a weak solution to anisotropic incompressible periodic Navier-Stokes system}

In this section, using the Galerkin approximation we show the existence of a weak solution of  the anisotropic Navier-Stokes system in the incompressible case, with general data in $L_2$-based Sobolev spaces on a flat torus ${\T}$, for $n\ge 2$. 

Let us consider the Navier-Stokes system
\begin{align}
\label{eq:mik33}
-&\bs{\mathfrak L}{\mathbf u}+\nabla p+({\mathbf u}\cdot \nabla ){\mathbf u}=\mathbf{f},\\
\label{eq:mik34}
&{\rm{div}}\, {\mathbf u}=0,
\end{align}
for the couple of unknowns $({\mathbf u},p )\in \dot{\mathbf H}_{\#}^{1}\times \dot H_\#^0 $ and the given data
$\mathbf f\in \dot{\mathbf H}_{\#}^{-1} $.
As for the Stokes system, the incompressible Navier-Stokes system \eqref{eq:mik33}-\eqref{eq:mik34} can be re-written as one vector equation
\begin{align}
\label{eq:mik35}
&-\bs{\mathfrak L}{\mathbf u}+\nabla p + ({\mathbf u}\cdot \nabla ){\mathbf u}=\mathbf{f}
\end{align}
for the  unknowns $({\mathbf u},p )\in \dot{\mathbf H}_{\#\sigma}^{1}\times \dot H_{q\#}^0 $, with some $q>1$, and the given data
$\mathbf f\in \dot{\mathbf H}_{\#}^{-1} $.

Next we show the existence of a weak solution of the Navier-Stokes equation, generalising to anisotropic case the Galerkin approximation arguments from \cite[Chapter 2]{Temam2001}, cf. also \cite[Chapter 1, Section 7]{Lions1969}.

First of all, let us define the space $\widetilde {\mathbf V}_{\#\sigma}$ and its norm as
\begin{align}\label{tildeV}
\widetilde {\mathbf V}_{\#\sigma}= \dot{\mathbf H}_{\#\sigma}^{1}\cap \mathbf L_{n\#},\qquad
\|\mathbf v\|_{\widetilde{\mathbf V}_{\#\sigma}}=\left(\|\mathbf v\|^2_{\dot{\mathbf H}_\#^1}+\| \mathbf v\|^2_{\mathbf L_{n\#}} \right)^{1/2}.
\end{align}
For the adjoint operator, we have
\begin{align}\label{tildeV*}
\widetilde {\mathbf V}_{\#\sigma}^*= (\dot{\mathbf H}_{\#\sigma}^{1}\cap \mathbf L_{n\#})^*
=(\dot{\mathbf H}_{\#\sigma}^{1})^*\cup \mathbf L_{n/(n-1)\#}. 
\end{align}
If $n\in\{2,3,4\}$, then $\widetilde {\mathbf V}_{\#\sigma}= \dot{\mathbf H}_{\#\sigma}^{1}$; 
otherwise $\widetilde {\mathbf V}_{\#\sigma}$ is a proper subspace of $\dot{\mathbf H}_{\#\sigma}^{1}$.
The space $\widetilde {\mathbf V}_{\#\sigma}$ is also the closure of $\dot{\mathbf C}^\infty_{\#\sigma}$ in the norm \eqref{tildeV}.
Taking into account the mapping properties of operator \eqref{eq:mik37d}, we give, similar to \cite[Chapter 2, Eq. (1.25)]{Temam2001} the following variational formulation of the Navier-Stokes system \eqref{eq:mik33}-\eqref{eq:mik34}, i.e., equation \eqref{eq:mik35}, for any $n\ge 2$:

{\it For $\mathbf f\in \dot{\mathbf H}_{\#}^{-1} $, find ${\mathbf u}\in \dot{\mathbf H}_{\#\sigma}^{1}$ such that} 
\begin{align}
\label{eq:mik51a}
\langle a_{ij}^{\alpha \beta }E_{j\beta }({\mathbf u}),E_{i\alpha }({\mathbf v})\rangle _{\T }
+\langle({\mathbf u}\cdot \nabla ){\mathbf u}, \mathbf v\rangle_{\T}
=\langle\mathbf{f}, \mathbf v\rangle_{\T}\quad\forall\,\mathbf v\in \widetilde {\mathbf V}_{\#\sigma}.
\end{align}

Since $\dot{\mathbf C}^\infty_{\#\sigma}\subset \widetilde {\mathbf V}_{\#\sigma}$, any ${\mathbf u}\in \dot{\mathbf H}_{\#\sigma}^{1}$ satisfying the variational problem \eqref{eq:mik51a} is also a distributional solution of the Navier-Stokes system \eqref{eq:mik33}-\eqref{eq:mik34} in the sense of Leray (i.e., for any $\mathbf v\in \dot{\mathbf C}^\infty_{\#\sigma}$ in  \eqref{eq:mik51a}).

Now we are in a position to prove the following assertion.

\begin{theorem}\label{th:mik4}
Let $n\ge 2$ and condition \eqref{eq:mik2} hold.
If \,$\mathbf f\in \dot{\mathbf H}_{\#}^{-1}$, then
the anisotropic Navier-Stokes equation \eqref{eq:mik35}
has a solution $({\mathbf u},p)\in \dot{\mathbf H}_{\#\sigma}^{1} \times \dot H_{q\#}^0 $ (in the sense of distributions), where $q=2$ for $n\in\{2,3,4\}$, and $q=n/(n-2)$ for $n\ge 5$.
Moreover, the following estimate holds
\begin{align}
\label{eq:mik51ae}
\| \mathbf u\|_{\dot{\mathbf H}_\#^1}
\le \pi^{-2}C_{\mathbb A}\|\mathbf f\|_{\dot{\mathbf H}_\#^{-1}}.
\end{align}
\end{theorem}
\begin{proof}
As in \cite[Chapter 2, Theorem 1.2]{Temam2001} and \cite[Chapter 1, Theorem 7.1]{Lions1969}, we will use the Galerkin approximation.
First of all, let $\mathbf w_1, \mathbf w_2,\ldots,\mathbf w_l,\ldots$ be a system of linearly independent functions from 
$\dot{\mathbf C}^\infty_{\#\sigma}$ that is complete in $\widetilde {\mathbf V}_{\#\sigma}$.

For each integer $m\ge 1$, let us look for a solution
\begin{align}\label{E5.14}
\mathbf u_m=\sum_{l=1}^m\eta_{l,m}\mathbf w_l, \quad \eta_{l,m}\in\R
\end{align}
of the following discrete analogue of the variational problem \eqref{eq:mik51a},
\begin{multline}
\label{eq:mik51b}
\langle a_{ij}^{\alpha \beta }E_{j\beta }({\mathbf u}_m),E_{i\alpha }(\mathbf w_k)\rangle _{\T }
+\langle({\mathbf u}_m\cdot \nabla ){\mathbf u}_m, \mathbf w_k\rangle_{\T}\\
=\langle\mathbf{f}, \mathbf w_k\rangle_{\T}\quad\forall\, k\in \{1,\ldots,m\}.
\end{multline}
For a fixed $m$, equations \eqref{eq:mik51b} give an algebraic system of nonlinear (quadratic) equations for $\eta_{l,m}$,  $l\in \{1,\ldots,m\}$. 
Existence of a real solution of this system follows from Lemma~\ref{BL-lemma4.3}.
Indeed, let $\eta:=\{\eta_{l,m}\}_{l=1}^m$, $Q(\eta):=\{Q_{k}(\eta)\}_{k=1}^m$ denote the $m$-dimensional vectors, where
$$
Q_{k}(\eta):=\langle a_{ij}^{\alpha \beta }E_{j\beta }({\mathbf u}_m),E_{i\alpha }(\mathbf w_k)\rangle _{\T }
+\langle({\mathbf u}_m\cdot \nabla ){\mathbf u}_m, \mathbf w_k\rangle_{\T}
-\langle\mathbf{f}, \mathbf w_k\rangle_{\T},
$$
and ${\mathbf u}_m={\mathbf u}_m(\eta)$ is given by \eqref{E5.14}.
Note that $E_{jj}({\mathbf u}_m)=0$ since $\div\, {\mathbf u}_m=0$.
Then by representation \eqref{E5.14}, equality \eqref{eq:mik55}, the ellipticity condition \eqref{eq:mik2}, the first Korn inequality \eqref{eq:mik15}, and the norm equivalence inequality \eqref{eq:mik14}, we obtain
\begin{align}\label{E5.16}
(Q(\eta),\eta)&=\langle a_{ij}^{\alpha \beta }E_{j\beta }({\mathbf u}_m),E_{i\alpha }(\mathbf u_m)\rangle _{\T }
+\langle({\mathbf u}_m\cdot \nabla ){\mathbf u}_m, \mathbf u_m\rangle_{\T}
-\langle\mathbf{f}, \mathbf u_m\rangle_{\T}
\nonumber\\
&=\langle a_{ij}^{\alpha \beta }E_{j\beta }({\mathbf u}_m),E_{i\alpha }(\mathbf u_m)\rangle _{\T }
-\langle\mathbf{f}, \mathbf u_m\rangle_{\T}
\nonumber\\
&\ge C_{\mathbb A}^{-1}\|{\mathbb E}({\mathbf u_m})\|_{(L_{2\#} )^{n\times n}}^2
-\|\mathbf f\|_{\dot{\mathbf H}_\#^{-1}}\|\mathbf u_m\|_{\dot{\mathbf H}_\#^1}
\nonumber\\
&\ge \pi^2C_{\mathbb A}^{-1}\| \mathbf u_m\|^2_{\dot{\mathbf H}_\#^1}
-\|\mathbf f\|_{\dot{\mathbf H}_\#^{-1}}\|\mathbf u_m\|_{\dot{\mathbf H}_\#^1}
\nonumber\\
&=(\pi^2C_{\mathbb A}^{-1}\| \mathbf u_m\|_{\dot{\mathbf H}_\#^1}
-\|\mathbf f\|_{\dot{\mathbf H}_\#^{-1}})\|\mathbf u_m\|_{\dot{\mathbf H}_\#^1}.
\end{align}
Thus $(Q(\eta),\eta)\ge 0\ \forall\eta: |\eta|=\rho$, where $\rho$ is sufficiently large (so that 
$
\|\mathbf u_m(\eta)\|_{\dot{\mathbf H}_\#^1}
\ge C_{\mathbb A}\pi^{-2}\|\mathbf f\|_{\dot{\mathbf H}_\#^{-1}}
$
$\forall\eta: |\eta|=\rho$).
Hence by Lemma~\ref{BL-lemma4.3} there exists $\eta=\{\eta_{l,m}\}_{l=1}^m$ such that $ |\eta|\le \rho$ and $Q(\eta)=0$, and then $\mathbf u_m(\eta)$ solves \eqref{eq:mik51b}.

Multiplying equations \eqref{eq:mik51b} by $\{\eta_{k,m}\}$ and summing them up in $k\in \{1,\ldots,m\}$, we obtain 
\begin{align}
\label{eq:mik51c}
\langle a_{ij}^{\alpha \beta }E_{j\beta }({\mathbf u}_m),E_{i\alpha }(\mathbf u_m)\rangle _{\T }
+\langle({\mathbf u}_m\cdot \nabla ){\mathbf u}_m, \mathbf u_m\rangle_{\T}
=\langle\mathbf{f}, \mathbf u_m\rangle_{\T}.
\end{align}
Similar to \eqref{E5.16}, this implies
\begin{align}
\label{eq:mik51d}
\pi^2C_{\mathbb A}^{-1}\| \mathbf u_m\|^2_{\dot{\mathbf H}_\#^1}
&\le C_{\mathbb A}^{-1}\|{\mathbb E}({\mathbf u_m})\|_{(L_{2\#} )^{n\times n}}^2
\nonumber\\
&\le\langle a_{ij}^{\alpha \beta }E_{j\beta }({\mathbf u}_m),E_{i\alpha }(\mathbf u_m)\rangle _{\T }
\nonumber\\
&=\langle\mathbf{f}, \mathbf u_m\rangle_{\T}
\le \|\mathbf f\|_{\dot{\mathbf H}_\#^{-1}}\|\mathbf u_m\|_{\dot{\mathbf H}_\#^1}.
\end{align}
Thus, for any $m=1,2,\ldots$
\begin{align}
\label{eq:mik51e}
\| \mathbf u_m\|_{\dot{\mathbf H}_\#^1}
\le \pi^{-2}C_{\mathbb A}\|\mathbf f\|_{\dot{\mathbf H}_\#^{-1}}.
\end{align}

This means that the sequence $\mathbf u_m$ is bounded in $\dot{\mathbf H}_\#^1$ and in $\dot{\mathbf H}_{\#\sigma}^1$ and thus there exists a subsequence $\mathbf u_{m'}$ weakly converging in $\dot{\mathbf H}_\#^1$ and in $\dot{\mathbf H}_{\#\sigma}^1$ to a function $\mathbf u\in\dot{\mathbf H}_{\#\sigma}^1$.
On the other hand, since $\dot{\mathbf H}_\#^1$ is compactly embedded in $\dot{\mathbf L}_{2\#}$ and $\dot{\mathbf H}_{\#\sigma}^1$ is compactly embedded in $\dot{\mathbf L}_{2\#\sigma}$, there exists a subsequence of $\mathbf u_{m'}$, for which we will use the same notation, that strongly converges in $\dot{\mathbf L}_{2\#}$ and $\dot{\mathbf L}_{2\#\sigma}$ to $\mathbf u$.

Then due to Lemma~\ref{L6.2} we can take limit in \eqref{eq:mik51b} as $m\to\infty$ to obtain
\begin{align}
\label{eq:mik51f}
\langle a_{ij}^{\alpha \beta }E_{j\beta }({\mathbf u}),E_{i\alpha }(\mathbf v)\rangle _{\T }
+\langle({\mathbf u}\cdot \nabla ){\mathbf u}, \mathbf v\rangle_{\T}
=\langle\mathbf{f}, \mathbf v\rangle_{\T}
\end{align}
for any $\mathbf v\in\{\mathbf w_k\}_{k=1}^\infty$.
Since by definition the set $\mathbf v\in\{\mathbf w_k\}_{k=1}^\infty$ is complete in $\widetilde {\mathbf V}_{\#\sigma}$ and operator \eqref{eq:mik37d} 
is bounded and continuous, we conclude that equation \eqref{eq:mik51f} holds for any $\mathbf v\in \widetilde {\mathbf V}_{\#\sigma}$, that is $\mathbf u$ solves variational problem \eqref{eq:mik51a} and moreover, \eqref{eq:mik51e} implies that $\mathbf u$ satisfies estimate \eqref{eq:mik51ae}.

After ${\mathbf u}\in \dot{\mathbf H}_{\#\sigma}^{1}$ satisfying \eqref{eq:mik51a} is obtained, the pressure $p\in  \dot{\mathcal D}'_\# $ can be found from equation \eqref{eq:mik33} re-written using notation \eqref{lllE} as  
\begin{align}
\label{eq:mik33a}
&\nabla p=\mathbf{f}-({\mathbf u}\cdot \nabla ){\mathbf u}+
\boldsymbol{\mathfrak L}{\mathbf u}
\end{align}
and understood in the sense of distributions. 
In the right hand side of \eqref{eq:mik33a}, $\mathbf f\in \dot{\mathbf H}_{\#}^{-1}$, 
$\boldsymbol{\mathfrak L}{\mathbf u}\in \dot{\mathbf H}_{\#}^{-1}$, while 
$({\mathbf u}\cdot \nabla ){\mathbf u}\in \dot{\mathbf H}_{\#}^{-1}$ if $n\in\{2,3,4\}$, and $({\mathbf u}\cdot \nabla ){\mathbf u}\in \dot{\mathbf L}_{n/(n-1)\#}$ if $n\ge 5$, cf. \eqref{eq:mik37c}, \eqref{eq:mik37d}. 
This implies that in fact $p\in \dot{L}_{2\#}$ if $n\in\{2,3,4\}$, and $p\in \dot{L}_{2\#}\cup\dot{H}^1_{n/(n-1)\#}\subset\dot{L}_{n/(n-2)\#}$ if $n\ge 5$, cf. \cite[Theorem IX.3.1, Remark IX.3.1]{Galdi2011}, \cite[Section 5.1]{RRS2016}.
\end{proof}

\subsection{Solution uniqueness for the anisotropic periodic Navier-Stokes system}
In this section we show that under additional constraint on the norm of the given data the weak solution of the Navier-Stokes equation \eqref{eq:mik35}  is unique.

For the uniqueness in the non-periodic setting,  
for the isotropic case \eqref{eq:mik8} with $\lambda=0$ and $\mu =1$ cf., e.g., \cite[Lemma 3.1]{Seregin2015};  
for the anisotropic case cf. \cite[Theorem 5.4]{KMW-DCDS2021}, \cite[Theorem 7.3]{KMW-transv2021}.
\begin{theorem}
\label{NSuniqueness}
Let $n\in\{2,3,4\}$ and condition \eqref{eq:mik2} hold.
Let $\mathbf f\in \dot{\mathbf H}_{\#}^{-1}$ and
\begin{align}
\label{uniqueness}
\|\mathbf f\|_{\dot{\mathbf H}_{\#}^{-1}}<\frac{\pi^3}{\sqrt{2}}C_{\mathbb A}^{-2}C_{4\#}^{-2}, 
\end{align}
with the constants $C_{\mathbb A}$ and $C_{4\#}$ from the ellipticity condition \eqref{eq:mik2} and the Sobolev embedding inequality \eqref{E7.9}, respectively.
Then
the anisotropic Navier-Stokes equation \eqref{eq:mik35}
has a unique solution $({\mathbf u},p)\in \dot{\mathbf H}_{\#\sigma}^{1} \times \dot H_{\#}^0$.
\end{theorem}
\begin{proof}
Assume that the Navier-Stokes equation \eqref{eq:mik35} and thus the variational problem \eqref{eq:mik51a} has two solutions $({\bf u}^{(1)},p^{(1)})$ and $({\bf u}^{(2)},p^{(2)})$ in the space 
$\dot{\mathbf H}_{\#\sigma}^{1}\times \dot H_{\#}^0$.
Note that  $\widetilde {\mathbf V}_{\#\sigma}= \dot{\mathbf H}_{\#\sigma}^{1}$ if $n\in\{2,3,4\}$.
Then  the first Korn inequality \eqref{eq:mik15}, the ellipticity condition \eqref{eq:mik2}, the variational formulation \eqref{eq:mik51a} and the norm equivalence inequality \eqref{eq:mik14} give for $k\in\{1,2\}$,
\begin{align}
\label{E6.27}
\|\nabla {\mathbf u^{(k)}}\|^2_{(L_{2\#})^{n\times n}}
&\le 2 \|{\mathbb E}({\mathbf u^{(k)}})\|_{(L_{2\#} )^{n\times n}}^2
\nonumber\\
&\le 2 C_{\mathbb A}\langle a_{ij}^{\alpha \beta }E_{j\beta }({\mathbf u}^{(k)}),E_{i\alpha }(\mathbf u^{(k)})\rangle _{\T }
\nonumber\\
&=2 C_{\mathbb A}\langle\mathbf{f}, \mathbf u^{(k)}\rangle_{\T}
\le 2 C_{\mathbb A}\|\mathbf f\|_{\dot{\mathbf H}_\#^{-1}}\|\mathbf u^{(k)}\|_{\dot{\mathbf H}_\#^1}
\nonumber\\
&\le \frac{\sqrt{2}}{\pi} C_{\mathbb A}\|\mathbf f\|_{\dot{\mathbf H}_\#^{-1}}\|\nabla {\mathbf u^{(k)}}\|_{(L_{2\#})^{n\times n}}.
\end{align}
Hence,
\begin{align}
\label{E6.28}
\|\nabla {\mathbf u^{(k)}}\|_{(L_{2\#})^{n\times n}}
&\le  \frac{\sqrt{2}}{\pi} C_{\mathbb A}\|\mathbf f\|_{\dot{\mathbf H}_\#^{-1}}, \quad k\in\{1,2\}.
\end{align}

The variational formulation \eqref{eq:mik51a} also implies,
\begin{multline}
\label{NS-var-eq-int-unique}
\big\langle a_{ij}^{\alpha \beta }E_{j\beta }({\bf u}^{(1)}-{\bf u}^{(2)}),E_{i\alpha }({\bf v })\big\rangle _\T\\
+\left\langle({\bf u}^{(1)}\cdot \nabla ){\bf u}^{(1)}-({\bf u}^{(2)}\cdot \nabla ){\bf u}^{(2)},{\bf v}\right\rangle _\T
=0 \quad \forall \ {\bf v}\in \dot{\mathbf H}_{\#\sigma}^{1}.
\end{multline}
Then by choosing ${\bf v}={\bf u}^{(1)}-{\bf u}^{(2)}$ in \eqref{NS-var-eq-int-unique}, we obtain
\begin{multline}
\label{NS-var-eq-int-unique-0}
\big\langle a_{ij}^{\alpha \beta }E_{j\beta }({\bf u}^{(1)}-{\bf u}^{(2)}),E_{i\alpha }({\bf u}^{(1)}-{\bf u}^{(2)})\big\rangle _\T\\
=-\left\langle\left(({\bf u}^{(1)}-{\bf u}^{(2)})\cdot \nabla )\right){\bf u}^{(1)},{\bf u}^{(1)}-{\bf u}^{(2)}\right\rangle _\T\\
-\left\langle({\bf u}^{(2)}\cdot \nabla )({\bf u}^{(1)}-{\bf u}^{(2)}),{\bf u}^{(1)}-{\bf u}^{(2)}\right\rangle _\T\,.
\end{multline}
Due to the membership of ${\bf u}^{(1)}$ and ${\bf u}^{(2)}$ in $\dot{\mathbf H}_{\#\sigma}^{1}$, relation \eqref{eq:mik55} yields
\begin{align}
\left\langle({\bf u}^{(2)}\cdot \nabla )({\bf u}^{(1)}-{\bf u}^{(2)}),{\bf u}^{(1)}-{\bf u}^{(2)}\right\rangle _\T=0\,,
\end{align}
which shows that equation \eqref{NS-var-eq-int-unique-0} reduces to
\begin{multline}
\label{NS-var-eq-int-uniqe-1}
\!\!\!\!\left\langle a_{ij}^{\alpha \beta }E_{j\beta }({\bf u}^{(1)}-{\bf u}^{(2)}),E_{i\alpha }({\bf u}^{(1)}-{\bf u}^{(2)})\right\rangle _\T\!\\
=\!-\left\langle\big( ({\bf u}^{(1)}-{\bf u}^{(2)})\cdot \nabla\big){\bf u}^{(1)},{\bf u}^{(1)}-{\bf u}^{(2)}\right\rangle _\T\,.
\end{multline}

On the other hand, in view of condition \eqref{eq:mik2} and the first Korn inequality \eqref{eq:mik15}, we deduce that
\begin{align}
\label{P-11-unique}
\|\nabla ({\bf u}^{(1)}-{\bf u}^{(2)})\|_{(L_{2\#})^{n\times n}}^2\, \leq
2C_{\mathbb A}\left\langle a_{ij}^{\alpha \beta }E_{j\beta }({\bf u}^{(1)}-{\bf u}^{(2)}),E_{i\alpha }({\bf u}^{(1)}-{\bf u}^{(2)})\right\rangle _\T\,.
\end{align}
Thus, by 
inequalities \eqref{eq:mik54-1b},  \eqref{E7.9}, \eqref{eq:mik13}, and \eqref{E6.28},
we obtain
\begin{align}
\label{NS-var-eq-int-uniqe-2}
&\left|\left\langle\left(({\bf u}^{(1)}-{\bf u}^{(2)})\cdot \nabla \right){\bf u}^{(1)},{\bf u}^{(1)}-{\bf u}^{(2)}\right\rangle _\T\right|
\nonumber\\
&\hspace{1em}\leq \|{\bf u}^{(1)}-{\bf u}^{(2)}\|^2_{\mathbf L_{4\#}}\|\nabla {\bf u}^{(1)}\|_{(L_{2\#})^{n\times n}}
\nonumber\\
&\hspace{1em}\leq C_{4\#}^2\|{\bf u}^{(1)}-{\bf u}^{(2)}\|_{{\mathbf H}_{\#}^{1}}^2
\|\nabla {\bf u}^{(1)}\|_{(L_{2\#})^{n\times n}}\nonumber\\
&\hspace{1em}\leq \frac{1}{\sqrt{2}\pi^3}C_{\mathbb A}C_{4\#}^2
\|\nabla ({\bf u}^{(1)}-{\bf u}^{(2)})\|_{(L_{2\#})^{n\times n}}^2 \|\mathbf f\|_{\dot{\mathbf H}_{\#}^{-1}}.
\end{align}
Then equalities \eqref{NS-var-eq-int-uniqe-1}-\eqref{NS-var-eq-int-uniqe-2} imply that
\begin{align}
\label{uniqueness-1}
\|\nabla ({\bf u}^{(1)}-{\bf u}^{(2)})\|_{(L_{2\#})^{n\times n}}^2
\leq \frac{\sqrt{2}}{\pi^3}C_{\mathbb A}^2C_{4\#}^2
\|\nabla ({\bf u}^{(1)}-{\bf u}^{(2)})\|_{(L_{2\#})^{n\times n}}^2 \|\mathbf f\|_{\dot{\mathbf H}_{\#}^{-1}}\,.
\end{align}
Assumption \eqref{uniqueness} shows that estimate \eqref{uniqueness-1} is possible only if ${\bf u}^{(1)}-{\bf u}^{(2)}={\bf 0}$.

Hence, equation \eqref{eq:mik35} implies $\nabla(p^{(1)}-p^{(2)})=0$ in $\Omega $. Then $p^{(1)}-p^{(2)}$ is a constant, i.e., $p^{(1)}=p^{(2)}$ in $\dot H_{\#}^0$.
\end{proof}

Note that in the second inequality in \eqref{NS-var-eq-int-uniqe-2} we used that by the Sobolev embedding theorem, 
$\|{\bf u}^{(1)}-{\bf u}^{(2)}\|_{\mathbf L_{4\#}}\leq C_{4\#}\|{\bf u}^{(1)}-{\bf u}^{(2)}\|_{{\mathbf H}_{\#}^{1}}$,
which is however not available for the dimensions $n>4$. This limited our uniqueness proof to the cases $n\in\{2,3,4\}$ only.

\subsection{Solution regularity for the anisotropic periodic Navier-Stokes system}

In this section, we show that the regularity of a solution of  the anisotropic incompressible Navier-Stokes system on ${\T}^n$, $n\in\{2,3,4\}\mg,$ is completely determined by the regularity of its right-hand side, as for the Stokes system.
To prove this we use the inclusions of the nonlinear term $({\mathbf u}^{(1)}\cdot \nabla ) {\mathbf u}$ given by Theorem~\ref{th:mik3} and the unique solvability of corresponding (linear) Stokes system along with the bootstrap argument.
This is sufficient to prove the regularity for $n\in\{2,3\}$.  However, to prove the regularity for $n=4$, we needed to accommodate more subtle norm estimates from \cite{Gerhardt1979} first.

\begin{theorem}\label{th:mik5}
Let $n\ge 2$ and condition \eqref{eq:mik2} hold.

(i) Let $s_1>{ -1+n/2}$. 
If $({\mathbf u},p)\in \dot{\mathbf H}_{\#\sigma}^{s_1} \times \dot H_\#^{s_1-1} $ is a solution of the anisotropic Navier-Stokes equation \eqref{eq:mik35} with a right hand side $\mathbf f\in \dot{\mathbf H}_{\#}^{s_2-2} $, where $s_2>s_1$,
then  $({\mathbf u},p)\in \dot{\mathbf H}_{\#\sigma}^{s_2} \times \dot H_\#^{s_2-1} $.

(ii) Moreover, if ${\mathbf f}\in\dot{\mathbf C}_\#^\infty$
then $({\mathbf u},p)\in\dot{\mathbf C}_\#^\infty\times \dot {\mathcal C}^\infty_\#$.
\end{theorem}
\begin{proof}
(i) Let $({\mathbf u},p)\in \dot{\mathbf H}_{\#\sigma}^{s_1} \times \dot H_\#^{s_1 -1} $ be a solution of  \eqref{eq:mik35} with  $\mathbf f\in \dot{\mathbf H}_{\#}^{s_2-2} $.
Then by Theorem~\ref{th:mik3}, for the nonlinear term we have the inclusion $({\mathbf u}^{(1)}\cdot \nabla )\mathbf u\in \dot{\mathbf H}_\#^{t_1}$ with $t_1=2s_1-1-n/2$ if $s_1<n/2$, with $t_1=s_1-1$ if $s_1>n/2$, and with any $t_1\in (s_1-2,s_1-1)$ (and we can further use $t_1=s_1-3/2$ for certainty) if $s_1=n/2$. 
Hence the couple $({\mathbf u},p)$ satisfies the equation
\begin{align}
\label{eq:mik53}
&-\bs{\mathfrak L}{\mathbf u}+\nabla p=\mathbf{f}^{(1)}
\end{align}
with $\mathbf f^{(1)}:=\mathbf f -({\mathbf u}^{(1)}\cdot \nabla )\mathbf u\in \dot{\mathbf H}_\#^{s^{(1)}-2}$, where $s^{(1)}=\min\{s_2, t_1+2\}$.
By Corollary~\ref{cor:mik1}(i), the linear equation \eqref{eq:mik53} has a unique solution in 
$\dot{\mathbf H}_{\#\sigma}^{s}\times \dot H_\#^{s-1}$ for any $s\le s^{(1)}$ and thus 
$({\mathbf u},p)\in \dot{\mathbf H}_{\#\sigma}^{s^{(1)}}\times \dot H_\#^{s^{(1)}-1}$.
If $s^{(1)}=s_2$, which we call case (a),  this proves item (i) of the theorem.

Otherwise we have case (b), when $s^{(1)}<s_2$,  i.e., $s^{(1)}=t_1+2$,
by the definition of $s^{(1)}$.
Then  we arrange an iterative process by replacing $s_1$ with $s^{(1)}=t_1+2$ on each iteration until we arrive at case (a), thus proving item (i) of the theorem.
Note that in case (b),
\begin{align*}
s^{(1)}-s_1\ge \delta:=\min\{s_1+1-n/2, 1, 1/2\}>0
\end{align*} 
in the first iteration, and $\delta$ does not decrease in the next iterations since  $s_1$ increases.
This implies that the iteration process will reach the case (a) and stop after a finite number of iterations. 

(ii) If ${\mathbf f}\in\dot{\mathbf C}_\#^\infty$, then for any $s_2\in\R$ we have $\mathbf f\in \dot{\mathbf H}_{\#}^{s_2-2} $  and item (i) implies that $({\mathbf u},p)\in \dot{\mathbf H}_{\#\sigma}^{s_2} \times \dot H_\#^{s_2-1} $. 
Hence $({\mathbf u},p)\in\dot{\mathbf C}_\#^\infty\times \dot {\mathcal C}^\infty_\#$.
\end{proof}

Combining Theorems \ref{th:mik4} and \ref{th:mik5}, we obtain the following assertion on existence and regularity of solution to the Navier-Stokes system on torus for $n\in \{2,3\}$.
\begin{theorem}\label{th:mik6}
Let $n\in \{2,3\}$ and condition \eqref{eq:mik2} hold.

(i) If $\mathbf f\in \dot{\mathbf H}_{\#}^{s-2} $, $s\ge 1$, then the anisotropic Navier-Stokes equation \eqref{eq:mik35} has a solution 
$({\mathbf u},p)\in \dot{\mathbf H}_{\#\sigma}^{s} \times \dot H_\#^{s-1} $.

(ii) Moreover, if ${\mathbf f}\in\dot{\mathbf C}_\#^\infty$ then equation \eqref{eq:mik35} has a solution 
$({\mathbf u},p)\in\dot{\mathbf C}_\#^\infty\times \dot {\mathcal C}^\infty_\#$.
\end{theorem} 

Note that in the {\em isotropic case} \eqref{eq:mik7} with $\lambda=0$, similar results for the Navier-Stokes system in flat torus as well as in domains of $\R^n$  are available, e.g., in \cite{Galdi2011, RRS2016, Seregin2015, Sohr2001,  Temam2001}.

One can easily check that the arguments leading to the regularity Theorem~\ref{th:mik6} are at this stage not applicable for $n\ge 4$ since to apply them we would need the existence of $({\mathbf u},p)$ in $\dot{\mathbf H}_{\#\sigma}^{s_1} \times \dot H_\#^{s_1-1} $ for $s_1>{ -1+n/2}$, which Theorem  \ref{th:mik4} does not provide.
Nevertheless, in the following assertion it appeared to be possible to   accommodate arguments of \cite{Gerhardt1979} 
to our anisotropic periodic setting  and  to prove the solution existence with $s_1=2$ and then the regularity for $n=4$.
\begin{theorem}\label{th:mik7}
Let $n\in \{2,3,4\}$  and condition \eqref{eq:mik2} hold.
If $\mathbf f\in \dot{\mathbf H}_{\#}^{0} $, then the anisotropic Navier-Stokes equation \eqref{eq:mik35} has a solution 
$({\mathbf u},p)\in \dot{\mathbf H}_{\#\sigma}^{2} \times \dot H_{\#}^{1}$
and the estimate 
\begin{align}
\label{eq:mik5.23}
\| \mathbf u\|_{\dot{\mathbf H}_{\#}^2} + \|p\|_{\dot{H}_{\#}^1}
\le  C_1\|\mathbf f \|_{\dot{\mathbf H}_{\#}^{0}} 
+C_{-1} \|\mathbf f\|^2_{\dot{\mathbf H}_{\#}^{-1}}
\end{align}
holds for some constants $C_1, C_{-1}\ge 0$.
\end{theorem}
\begin{proof}
Let $({\mathbf u},p)\in \dot{\mathbf H}_{\#\sigma}^{1} \times \dot H_{\#}^{0}$ be the solution of \eqref{eq:mik35} provided by Theorem~\ref{th:mik4} and thus satisfying \eqref{eq:mik51ae}. 
Let $\mathbf u_k\in\dot {\mathbf C}^\infty_{\#\sigma}$ be a sequence 
converging to $\mathbf u$ in  $\dot{\mathbf H}_{\#\sigma}^{1}$ and  let us consider the following Oseen equation ({\mg a} linearised version of equation \eqref{eq:mik35}) for 
$(\widetilde{\mathbf  u}_k,\widetilde{p}_k)$,
\begin{align}
\label{eq:mik35lin}
&-\bs{\mathfrak L}\widetilde{\mathbf  u}_k+\nabla \widetilde{p}_k
=\mathbf{f}-({\mathbf u}_k\cdot \nabla )\widetilde{\mathbf  u}_k.
\end{align}
By Theorem \ref{th:mik1-Os-inf}, for every $k$ there exists a solution $(\widetilde{\mathbf  u}_k,\widetilde{p}_k)\in \dot{\mathbf H}_{\#\sigma}^{2} \times \dot H_{\#}^{1}$ of the linear Oseen equation \eqref{eq:mik35lin}. 
Moreover, since $\mathbf u_k\in\dot {\mathbf C}^\infty_{\#\sigma}$, estimate \eqref{estimate-1-wp-S-2u} from Theorem \ref{Oseen-problemTh} implies that
\begin{align}
\label{eq:mik51ek}
\| \widetilde{\mathbf  u}_k\|_{\dot{\mathbf H}_{\#}^1}
\le C_{\mathbb A}\pi^{-2}\|\mathbf f\|_{\dot{\mathbf H}_\#^{-1}}.
\end{align}

In addition, considering \eqref{eq:mik35lin} as a linear anisotropic Stokes equation with a given right hand side, by Corollary \ref{cor:mik1} we obtain the estimate
\begin{align}\label{E5.25}
\|\widetilde{\mathbf  u}_k\|_{\dot{\mathbf H}_{\#\sigma}^{2}}+\|\widetilde{p}_k \|_{\dot H_{\#}^{1}}
\leq C\left(\|\mathbf f \|_{\dot{\mathbf H}_{\#}^{0}} 
+\|({\mathbf u}_k\cdot \nabla )\widetilde{\mathbf  u}_k \|_{\dot{\mathbf H}_{\#}^{0}}\right)
\end{align}
with $C=C_{uf}+C_{pf}$.

Let $\theta\in\R$ be such that $2<\theta<\infty$ if $n=2$, while  $\theta=n$ if $n\in\{3,4\}$.
Employing Lemma \ref{L6.5} with ${\mathbf u}_k$ for $u$, the sequence $\{{\mathbf u}_k\}$ for $K_{\theta\#}$, and $\nabla\widetilde{\mathbf  u}_k$ for $w$, we obtain by \eqref{eq:mik14} and \eqref{eq:mik51ek} that
for any $\delta>0$ there exists a constant $\widetilde C_\delta{\mg =C_\delta(K_{\theta\#})}>0$ such that
\begin{multline}\label{E5.28}
\|({\mathbf u}_k\cdot \nabla )\widetilde{\mathbf  u}_k \|_{\dot{\mathbf H}_{\#}^{0}}
\le\delta  \|\nabla\widetilde{\mathbf  u}_k\|_{(H^1_{\#})^{n\times n}}
+\widetilde C_\delta\|{\mathbf u}_k\|_{\mathbf L_{\theta\#}} \|\nabla\widetilde{\mathbf  u}_k\|_{(L_{2\#})^{n\times n}}\\
\le 2\pi\delta \|\widetilde{\mathbf  u}_k\|_{\dot{\mathbf H}^2_{\#}}
+2\pi \widetilde C_\delta\|{\mathbf u}_k\|_{\mathbf L_{\theta\#}}  \|\widetilde{\mathbf  u}_k\|_{\dot{\mathbf H}^1_{\#}}\\
\le 2\pi\delta  \|\widetilde{\mathbf  u}_k\|_{\dot{\mathbf H}^2_{\#}}
+2\pi^{-1}C_{\mathbb A}C_{\theta\#} \widetilde C_\delta\|{\mathbf u}_k\|_{\dot{\mathbf H}^1_{\#}}\|\mathbf f\|_{\dot{\mathbf H}_\#^{-1}}.
\end{multline}
Here we also took into account 
that for $2\le n \le 4$ there exists a constant $C_{\theta\#}>0$ independent of $\mathbf v$, such that
$
\|{\mathbf v}\|_{\mathbf L_{\theta\#}} \le C_{\theta\#}\|{\mathbf v}\|_{{\mathbf H}^1_{\#}}
$
for any  $\mathbf v\in {\mathbf H}^1_{\#}$, due to the Sobolev embedding theorem.

Since ${\mathbf u}_k$ converges to ${\mathbf u}$ in $\dot{\mathbf H}_{\#\sigma}^{1}$, estimate \eqref{eq:mik51ae} means that there exists $k_0\in \N$ such that for $k\ge k_0$ (which we will further assume)
$$
\|{\mathbf u}_k\|_{\dot{\mathbf H}^1_{\#}}\le 2\|{\mathbf u}\|_{\dot{\mathbf H}^1_{\#}}
\le 2\pi^{-2}C_{\mathbb A}\|\mathbf f\|_{\dot{\mathbf H}_\#^{-1}}
$$
and hence \eqref{E5.28} implies
\begin{align}\label{E6.21}
\|({\mathbf u}_k\cdot \nabla )\widetilde{\mathbf  u}_k \|_{\dot{\mathbf H}_{\#}^{0}}
\le 2\pi\delta  \|\widetilde{\mathbf  u}_k\|_{\dot{\mathbf H}^2_{\#}}
+4\pi^{-3}C_{\mathbb A}^2C_{\theta\#} \widetilde C_\delta \|\mathbf f\|^2_{\dot{\mathbf H}_\#^{-1}}.
\end{align}
Substituting \eqref{E6.21} in \eqref{E5.25} with $\delta=1/(4\pi C)$ 
we obtain
\begin{align}\label{E5.25a}
\|\widetilde{\mathbf  u}_k\|_{\dot{\mathbf H}_{\#\sigma}^{2}}+\|\widetilde{p}_k \|_{\dot H_{\#}^{1}}
\leq 2C\|\mathbf f \|_{\dot{\mathbf H}_{\#}^{0}} 
+  C_{-1} \|\mathbf f\|^2_{\dot{\mathbf H}_{\#}^{-1}},
\end{align}
where $C_{-1}:=8\pi^{-3}CC_{\mathbb A}^2C_{\theta\#} \widetilde C_\delta $.

Together with \eqref{eq:mik35lin} this implies that there exist subsequences of $\{\widetilde{\mathbf  u}_k\}$ and $\{\widetilde{p}_k\}$ weakly converging in 
${\dot{\mathbf H}_{\#\sigma}^{2}}$ and $\dot H_{\#}^{1}$, respectively, 
to a weak solution
$(\widetilde{\mathbf  u},\widetilde{p})\in {\dot{\mathbf H}_{\#\sigma}^{2}}\times\dot H_{\#}^{1}\subset \dot{\mathbf H}_{\#\sigma}^{1} \times \dot H_{\#}^{0}$ 
of the Oseen equation  
\begin{align}
\label{eq:mik35lim}
&-\bs{\mathfrak L}\widetilde{\mathbf  u}+\nabla \widetilde{p}=\mathbf{f}-({\mathbf u}\cdot \nabla )\widetilde{\mathbf  u}
\end{align} 
satisfying the estimate
\begin{align}\label{E5.25b}
\|\widetilde{\mathbf  u}\|_{\dot{\mathbf H}_{\#\sigma}^{2}}+\|\widetilde{p} \|_{\dot H_{\#}^{1}}
\leq 2C\|\mathbf f \|_{\dot{\mathbf H}_{\#}^{0}} 
+C_{-1} \|\mathbf f\|^2_{\dot{\mathbf H}_{\#}^{-1}}.
\end{align}

Since $({\mathbf u},p)\in \dot{\mathbf H}_{\#\sigma}^{1} \times \dot H_{\#}^{0}$ satisfy \eqref{eq:mik35}, the solution uniqueness for the Oseen equation \eqref{eq:mik35lim} with a fixed $\mathbf u$ in 
${\dot{\mathbf H}_{\#\sigma}^{1}}\times\dot H_{\#}^{0}$ (see Theorem \ref{Oseen-problemTh}(ii)) implies that 
$({\mathbf u},p)=(\widetilde{\mathbf  u},\widetilde{p})\in{\dot{\mathbf H}_{\#\sigma}^{2}}\times\dot H_{\#}^{1}$ and estimate \eqref{eq:mik5.23} holds with $C_1=2C$.
\end{proof}

Note that in the third inequality in \eqref{E5.28} we used the Sobolev embedding theorem that implied that
there exists a constant $C_{\theta\#}>0$ independent of $\mathbf v$, such that
$
\|{\mathbf v}\|_{\mathbf L_{\theta\#}} \le C_{\theta\#}\|{\mathbf v}\|_{{\mathbf H}^1_{\#}}
$
for any  $\mathbf v\in {\mathbf H}^1_{\#}$.
This estimate is however not available for the dimensions $n>4$, which limited our regularity proof in Theorem \ref{th:mik7} to the cases $n\in\{2,3,4\}$ only.

Combining Theorems \ref{th:mik7} and \ref{th:mik5}, we obtain the following assertion on existence and regularity of solution to the Navier-Stokes system on torus for $n\in\{2,3,4\}$. Note that for $n\in\{2,3\}$ the assertion is already obtained in Theorem \ref{th:mik6}.

\begin{theorem}\label{th:mik8}
Let $n\in\{2,3,4\}$ and condition \eqref{eq:mik2} hold.

(i) If $\mathbf f\in \dot{\mathbf H}_{\#}^{s-2} $, $s\ge 2$, then the anisotropic Navier-Stokes equation \eqref{eq:mik35} has a solution 
$({\mathbf u},p)\in \dot{\mathbf H}_{\#\sigma}^{s} \times \dot H_\#^{s-1} $.

(ii) Moreover, if ${\mathbf f}\in\dot{\mathbf C}_\#^\infty$ then equation \eqref{eq:mik35} has a solution 
$({\mathbf u},p)\in\dot{\mathbf C}_\#^\infty\times \dot {\mathcal C}^\infty_\#$.
\end{theorem}

\section{Some auxiliary results}\label{sec:mik6}
\subsection{\bf Abstract mixed variational formulations}
\label{B-B-theory}
Let us produce the well-posedness result for the abstract mixed formulation, related to  Babu\v{s}ka \cite{Babuska1973} and Brezzi \cite[Theorem 1.1]{Brezzi1974}
(see also Theorem 2.34 and Remark 2.35(i) in Ern \& Guermond \cite{Ern-Guermond2004}, and Brezzi \& Fortin \cite{Brezzi-Fortin1991}).
\begin{theorem}\label{B-B}
Let $X$ and ${\mathcal M}$ be two real Hilbert spaces. Let $a(\cdot ,\cdot):X\times X\to {\mathbb R}$ and $b(\cdot ,\cdot):X\times {\mathcal M}\to {\mathbb R}$ be bounded bilinear forms. Let $f\in X'$ and $g\in {\mathcal M}'$. Let $V$ be the subspace of $X$ defined by
\begin{align}
\label{V}
V:=\left\{v\in X: b(v,q)=0\quad \forall \, q\in {\mathcal M}\right\}.
\end{align}
Assume that $a(\cdot ,\cdot ):V\times V\to {\mathbb R}$ is coercive, which means that there exists a constant $C_a>0$ such that
\begin{align}
\label{coercive}
a(w,w)\geq C_a^{-1}\|w\|_X^2\quad \forall \, w\in V,
\end{align}
and that $b(\cdot ,\cdot):X\times {\mathcal M}\to {\mathbb R}$ satisfies
the Babu\v{s}ka-Brezzi condition
\begin{align}
\label{inf-sup-sm}
&\inf _{q\in {\mathcal M}\setminus \{0\}}\sup_{v\in X\setminus \{0\}}\frac{b(v,q)}{\|v\|_X\|q\|_{\mathcal M}}\geq C_b^{-1} \,,
\end{align}
with some constant $C_b>0$. Then the mixed variational {formulation}
\begin{equation}
\label{mixed-variational}
\left\{\begin{array}{ll}
a(u,v)+b(v,p)\hspace{-0.5em}&=f(v) \quad \forall \, v\in X,\\
b(u,q)&=g(q) \quad \forall \, q\in {\mathcal M},
\end{array}
\right.
\end{equation}
has a unique solution $(u,p)\in X\times {\mathcal M}$ and
\begin{align}
\label{mixed-Cu}
&\|u\|_{X}\leq C_a\|f\|_{X'}+C_b(1+\|a\|C_a)\|g\|_{\mathcal M'},\\
\label{mixed-Cp}
&\|p\|_{{\mathcal M}}\leq C_b(1+\|a\|C_a)\|f\|_{X'}+{\|a\|}C_b^{2}(1+\|a\|C_a)\|g\|_{\mathcal M'},
\end{align}
{where $\|a\|$ is the norm of the bilinear form $a(\cdot ,\cdot )$.}
\end{theorem}

\subsection{Brower fixed point theorem application} 
In the main text we need the following well-known result that follows from the Brower fixed point theorem (see, e.g., \cite[{\mg Chapter 1,} Lemma 4.3]{Lions1969}, \cite[Lemma IX.3.1]{Galdi2011}).
\begin{lemma}\label{BL-lemma4.3}
Let $\eta\to Q(\eta)$ be a continuous map of $\R^m$ to itself, such that for some $\rho>0$, 
\begin{align}
(Q(\eta),\eta)\ge 0\quad \forall\eta:\ |\eta|=\rho.
\end{align}
Here for $\eta=\{\eta_j\}, \zeta=\{\zeta_j\}\in\R^m$ we denote
$$
(\eta,\zeta):=\sum_{j=1}^m\eta_j\zeta_j,\quad |\eta|:=(\eta,\eta)^{1/2}.
$$
Then there exists $\eta$ such that $|\eta|\le\rho$ and $Q(\eta)=0$.
\end{lemma}

\subsection{Advection term properties}\label{S7.2}

Let the quadratic operator $\bs{\mathit B}: {\mathbf w}\mapsto \bs{\mathit B}\mathbf w$ be defined as 
$
\bs{\mathit B}{\mathbf w}:=({\mathbf w}\cdot \nabla ){\mathbf w}.
$

Let in this section the dimension $n\ge 2$. 
To formulate assertions valid both for $n=2$ and $n>2$, let us define the set $I_n$ as 
\begin{align}\label{In}
I_2:=(2,\infty),\quad I_n:=[n,\infty) \mbox{ if }n>2.
\end{align}
{\mg Let $\theta\in I_n$ and let us denote $q_\theta:=2\theta/(\theta-2)$. Then}
\begin{align}\label{2theta}
2<2\theta/(\theta-2)<\infty \mbox{ if } n=2;\quad 2<2\theta/(\theta-2)\le 2n/(n-2) \mbox{ if } n>2.
\end{align}
By the Sobolev embedding theorem (see, e.g., \cite[Section 2.2.4, Corollary 2]{Runst-Sickel1996}), for any $\theta\in I_n$
the space ${\mathbf H}_{\#}^{1}$ is continuously embedded in the space $\mathbf L_{2\theta/(\theta-2)\#}$
and there exists a constant $C_{2\theta/(\theta-2)\#}>0$ independent of $\mathbf v$, such that
\begin{align}\label{E7.9}
\|{\mathbf v}\|_{\mathbf L_{2\theta/(\theta-2)\#}} \le C_{2\theta/(\theta-2)\#}\|{\mathbf v}\|_{{\mathbf H}^1_{\#}}
\quad
\forall\  \mathbf v\in {\mathbf H}^1_{\#},\ \forall\,\theta\in I_n .
\end{align}

\subsubsection{} \label{S7.2.2}
By the H\"older inequality, for any $\theta\in I_n$
\begin{multline}\label{eq:mik54-1}
|\langle({\mathbf v}_1\cdot \nabla ){\mathbf v}_2,{\mathbf v}_3\rangle _{\T}|
\!\!\le\!\|{\mathbf v}_1\|_{\mathbf L_{2\theta/(\theta-2)\#}} \|\nabla {\mathbf v}_2\|_{(L_{2\#})^{n\times n}} 
\|{\mathbf v}_3\|_{\mathbf L_{\theta\#}}\\
 \forall\ {\mathbf v}_1\in \mathbf L_{2\theta/(\theta-2)\#},\
{\mathbf v}_2\in {\mathbf H}_{\#}^{1},\ 
{\mathbf v}_3\in \mathbf L_{\theta\#},\ \forall\,\theta\in I_n.
\end{multline}
Due to  \eqref{E7.9} {\cn and \eqref{eq:mik14}}, inequality \eqref{eq:mik54-1} gives
\begin{multline}\label{eq:mik54-2}
|\langle({\mathbf v}_1\cdot \nabla ){\mathbf v}_2,{\mathbf v}_3\rangle _{\T}|
\!\!\le\!{\cn 2\pi}C_{2\theta/(\theta-2)\#}\|{\mathbf v}_1\|_{{\mathbf H}_{\#}^{1}} \|{\mathbf v}_2\|_{{\mathbf H}_{\#}^{1}} 
\|{\mathbf v}_3\|_{\mathbf L_{\theta\#}}\\
 \forall\ {\mathbf v}_1,{\mathbf v}_2\in {\mathbf H}_{\#}^{1},\ 
{\mathbf v}_3\in \mathbf L_{\theta\#},\ \forall\,\theta\in I_n.
\end{multline}
This implies that the trilinear form $\langle({\mathbf v}_1\cdot \nabla ){\mathbf v}_2,{\mathbf v}_3\rangle _{\T}$ is bounded and continuous on ${\mathbf H}_{\#}^{1}\times{\mathbf H}_{\#}^{1}\times  \mathbf L_{\theta\#},\ \forall\,\theta\in I_n$. 
Taking into account that the space $\mathbf L_{\theta/(\theta-1)\#}$ is dual to the space $\mathbf L_{\theta\#}$,
this implies that if ${\mathbf v}_1,{\mathbf v}_2\in {\mathbf H}_{\#}^{1}$, then
$$
({\mathbf v}_1\cdot \nabla ){\mathbf v}_2\in{\mathbf L_{\theta/(\theta-1)\#}}
,\ \forall\,\theta\in I_n
$$
and the following estimate holds for the bilinear {\cn term}
\begin{align*}
\|({\mathbf v}_1\cdot \nabla ){\mathbf v}_2\|_{\mathbf L_{\theta/(\theta-1)\#}}
\le {\cn 2\pi} C_{2\theta/(\theta-2)\#} \|{\mathbf v}_1\|_{\mathbf H^1_{\#}} \|{\mathbf v}_2\|_{\mathbf H^1_{\#}},\ \forall\,\theta\in I_n.
\end{align*}
This means that the quadratic operator
\begin{align}
\label{eq:mik37a}
&{\bs{\mathit B}}:{\mathbf H}^{1}_{\#}\to \mathbf L_{\theta/(\theta-1)\#},\ \forall\,\theta\in I_n
\end{align} 
is bounded and continuous. 

Moreover, if ${\mathbf v}_1\in{\mathbf H}_{\#\sigma}^{1}$ and ${\mathbf v}_2\in{\mathbf H}_{\#}^{1}$, then
\begin{align}\label{E7.19}
\int_{\T}({\mathbf v}_1\cdot \nabla ){\mathbf v}_2dx
=\int_{\T}\partial_i({v}_{1,i}{\mathbf v}_{2})dx - \int_{\T}(\div\, {\mathbf v}_1 ){\mathbf v}_{2}dx=0,
\end{align}
and hence
$
({\mathbf v}_1\cdot \nabla ){\mathbf v}_2\in\dot{\mathbf L}_{\theta/(\theta-1)\#} ,\ \forall\,\theta\in I_n, 
$
implying the boundedness and continuity of the quadratic operator
\begin{align}
\label{eq:mik37c}
&{\bs{\mathit B}}:\dot{\mathbf H}_{\#\sigma}^{1}\to \dot{\mathbf L}_{\theta/(\theta-1)\#},\ \forall\,\theta\in I_n.
\end{align}
For $n\ge 3$, taking $\theta=n$,
this particularly implies that the quadratic operator
\begin{align}
\label{eq:mik37d}
&{\bs{\mathit B}}:\dot{\mathbf H}_{\#\sigma}^{1}\to\widetilde {\mathbf V}_{\#\sigma}^*
=(\dot{\mathbf H}_{\#\sigma}^{1})^*\cup \mathbf L_{n/(n-1)\#}.
\end{align}
is also bounded and continuous.
In fact, operator \eqref{eq:mik37d} is bounded and continuous also for $n=2$.

Indeed, for $n=2$, estimate \eqref{eq:mik54-2} with $\theta=4$ leads to the estimate
\begin{multline}\label{eq:mik54-2n2}
|\langle({\mathbf v}_1\cdot \nabla ){\mathbf v}_2,{\mathbf v}_3\rangle _{\T}|
\!\!\le\!{\cn 2\pi}C_{4\#}\|{\mathbf v}_1\|_{{\mathbf H}_{\#}^{1}} \|{\mathbf v}_2\|_{{\mathbf H}_{\#}^{1}} 
\|{\mathbf v}_3\|_{\mathbf L_{4\#}}\\
 \forall\ {\mathbf v}_1,{\mathbf v}_2\in {\mathbf H}_{\#}^{1},\ 
{\mathbf v}_3\in \mathbf L_{4\#}.
\end{multline}
Assume now that ${\mathbf v}_3\in {\mathbf H}_{\#}^{1}$ and again take into account that by the Sobolev embedding theorem 
the space ${\mathbf H}_{\#}^{1}$ is continuously embedded in the space $\mathbf L_{4\#}$
and 
\begin{align}\label{E7.9n2}
\|{\mathbf v}_3\|_{\mathbf L_{4\#}} \le C_{4\#}\|{\mathbf v}_3\|_{{\mathbf H}^1_{\#}}
\quad
\forall\  \mathbf v_3\in {\mathbf H}^1_{\#},\ n=2.
\end{align}
Due to  \eqref{E7.9n2}, inequality \eqref{eq:mik54-2n2} gives
\begin{multline*}
|\langle({\mathbf v}_1\cdot \nabla ){\mathbf v}_2,{\mathbf v}_3\rangle _{\T}|
\!\!\le\!{\cn 2\pi}C^2_{4\#}\|{\mathbf v}_1\|_{{\mathbf H}_{\#}^{1}} \|{\mathbf v}_2\|_{{\mathbf H}_{\#}^{1}} 
\|{\mathbf v}_3\|_{{\mathbf H}_{\#}^{1}}\\
 \forall\ {\mathbf v}_1,{\mathbf v}_2,{\mathbf v}_3\in {\mathbf H}_{\#}^{1},\ n=2.
\end{multline*}
This implies that the trilinear form $\langle({\mathbf v}_1\cdot \nabla ){\mathbf v}_2,{\mathbf v}_3\rangle _{\T}$ is bounded and continuous on ${\mathbf H}_{\#}^{1}\times{\mathbf H}_{\#}^{1}\times  {\mathbf H}_{\#}^{1}$. 
Taking into account that the space ${\mathbf H}_{\#}^{-1}$ is dual to the space ${\mathbf H}_{\#}^{1}$,
this implies that if ${\mathbf v}_1,{\mathbf v}_2\in {\mathbf H}_{\#}^{1}$, then
$$
({\mathbf v}_1\cdot \nabla ){\mathbf v}_2\in{\mathbf H}_{\#}^{-1},\ n=2,
$$
and the following estimate holds for the bilinear operator
\begin{align*}
\|({\mathbf v}_1\cdot \nabla ){\mathbf v}_2\|_{{\mathbf H}_{\#}^{-1}}
\le {\cn 2\pi}C^2_{4\#} \|{\mathbf v}_1\|_{\mathbf H^1_{\#}} \|{\mathbf v}_2\|_{\mathbf H^1_{\#}},\ n=2.
\end{align*}
This means that the quadratic operator
\begin{align*}
&{\bs{\mathit B}}:{\mathbf H}^{1}_{\#}\to{\mathbf H}_{\#}^{-1},\ n=2
\end{align*} 
is bounded and continuous. 
Moreover, if ${\mathbf v}_1\in{\mathbf H}_{\#\sigma}^{1}$ and ${\mathbf v}_2\in{\mathbf H}_{\#}^{1}$, then by \eqref{E7.19}
we also obtain
$
({\mathbf v}_1\cdot \nabla ){\mathbf v}_2\in\dot{\mathbf H}_{\#}^{-1} 
$
and the continuity of the quadratic operator
\begin{align}
\label{eq:mik37cn2}
&{\bs{\mathit B}}:\dot{\mathbf H}_{\#\sigma}^{1}\to \dot{\mathbf H}_{\#}^{-1},\ n=2.
\end{align}
Since ${\mathbf H}_{\#}^{-1}=({\mathbf H}_{\#}^{1})^*\subset({\mathbf H}_{\#\sigma}^{1})^*
\subset(\dot{\mathbf H}_{\#\sigma}^{1})^*\cup \mathbf L_{n/(n-1)\#}
=\widetilde {\mathbf V}_{\#\sigma}^*
$, the continuity of operator \eqref{eq:mik37cn2}
implies also the boundedness and continuity of operator
\begin{align*}
&{\bs{\mathit B}}:\dot{\mathbf H}_{\#\sigma}^{1}\to\widetilde {\mathbf V}_{\#\sigma}^*, \ n=2,
\end{align*}
that is, of operator \eqref{eq:mik37d} also for $n=2$.

\subsubsection{} 
Let us also give another well known result, which proof can be easily accommodated to the periodic case from the one  found, e.g., in \cite[{\mg Chapter} 2, Lemma 1.5]{Temam2001}, see also \cite[{\mg Chapter 1,} Teorem 7.1]{Lions1969}.
\begin{lemma}\label{L6.2}
Let $\mathbf u_k$ converges to $\mathbf u$ weakly in $\dot{\mathbf H}_{\#\sigma}^1$  and strongly in $\dot{\mathbf L}_{2\#\sigma}$.
Then 
$
\langle\bs{\mathit B}{\mathbf u}_k, \mathbf v\rangle_{\T}\to \langle({\mathbf u}\cdot \nabla ){\mathbf u}, \mathbf v\rangle_{\T}\quad
\forall\, \mathbf v\in  \dot{\mathbf C}^\infty_{\#\sigma}.
$
\end{lemma}

\subsubsection{}
Similar to \eqref{eq:mik54-1}, we  have,
\begin{multline}\label{eq:mik54-1b}
|\langle({\mathbf v}_1\cdot \nabla ){\mathbf v}_2,{\mathbf v}_3\rangle _{\T}|
\!\!\le\!\|{\mathbf v}_1\|_{\mathbf L_{\theta\#}} \|\nabla {\mathbf v}_2\|_{(L_{2\#})^{n\times n}} 
\|{\mathbf v}_3\|_{\mathbf L_{2\theta/(\theta-2)\#}}\\
 \forall\ {\mathbf v}_1\in \mathbf L_{\theta\#},\
{\mathbf v}_2\in {\mathbf H}_{\#}^{1},\ 
{\mathbf v}_3\in \mathbf L_{2\theta/(\theta-2)\#}.
\end{multline}
Due {\mg to} \eqref{E7.9}, inequality \eqref{eq:mik54-1b} implies
\begin{multline}\label{eq:mik54-2a}
|\langle({\mathbf v}_1\cdot \nabla ){\mathbf v}_2,{\mathbf v}_3\rangle _{\T}|
\!\!\le{\cn 2\pi}C_{2\theta/(\theta-2)\#}\|{\mathbf v}_1\|_{\mathbf L_{\theta\#}} \|{\mathbf v}_2\|_{{\mathbf H}_{\#}^{1}} 
\|{\mathbf v}_3\|_{{\mathbf H}_{\#}^{1}}\\
 \forall\ \mathbf v_1\in \mathbf L_{\theta\#},\
 \forall\  \mathbf v_2, \mathbf v_3\in {\mathbf H}_{\#}^{1}.
\end{multline}

This implies that if ${\mathbf v}_1\in \mathbf L_{\theta\#}$, ${\mathbf v}_2\in {\mathbf H}_{\#}^{1}$, then
$
({\mathbf v}_1\cdot \nabla ){\mathbf v}_2\in {\mathbf H}_{\#}^{-1}
$
and the following estimate holds for the bilinear operator
\begin{align}
\label{eq:mik37a2a}
&\|({\mathbf v}_1\cdot \nabla ){\mathbf v}_2\|_{\mathbf H^{-1}_{\#}}
\le {\cn 2\pi}C_{2\theta/(\theta-2)\#} \|{\mathbf v}_1\|_{\mathbf L_{\theta\#}} \|{\mathbf v}_2\|_{\mathbf H^1_{\#}}.
\end{align}
Moreover,  if   ${\mathbf v}_1\in \mathbf L_{\theta\#\sigma}$, ${\mathbf v}_2\in {\mathbf H}_{\#}^{1}$, 
then again \eqref{E7.19} holds, implying that  by \eqref{eq:mik37a2a},
\begin{align}
\label{eq:mik37c2a}
&({\mathbf v}_1\cdot \nabla ){\mathbf v}_2\in \dot{\mathbf H}_{\#}^{-1},\nonumber\\
&\|({\mathbf v}_1\cdot \nabla ){\mathbf v}_2\|_{\dot{\mathbf H}^{-1}_{\#}}
\le {\cn 2\pi}C_{2\theta/(\theta-2)\#} \|{\mathbf v}_1\|_{\mathbf L_{\gr\theta\#}} \|{\mathbf v}_2\|_{{\mathbf H}^1_{\#}}.
\end{align}

\subsubsection{}\label{S7.3.3}  
The  divergence theorem and periodicity
imply the following identity for any ${\mathbf v}_1,{\mathbf v}_2,{\mathbf v}_3\in \mathbf C^\infty_\#$.
\begin{align}
\label{eq:mik54}
\left\langle({\mathbf v}_1\cdot \nabla ){\mathbf v}_2,{\mathbf v}_3\right\rangle _{\T}
&=\int_{\T}\nabla\cdot\left({\mathbf v}_1({\mathbf v}_2\cdot {\mathbf v}_3)\right)d{\mathbf x}
-\left\langle(\nabla \cdot{\mathbf v}_1){\mathbf v}_3
+({\mathbf v}_1\cdot \nabla ){\mathbf v}_3,{\mathbf v}_2\right\rangle _{\T}\nonumber
\\
&=
-\left\langle({\mathbf v}_1\cdot \nabla ){\mathbf v}_3,{\mathbf v}_2\right\rangle _{\T}
-\left\langle(\nabla \cdot{\mathbf v}_1){\mathbf v}_3,{\mathbf v}_2\right\rangle _{\T}
 \end{align}
In view of \eqref{eq:mik54} we obtain the identity
\begin{align*}
\left\langle({\mathbf v}_1\cdot \nabla ){\mathbf v}_2,{\mathbf v}_3\right\rangle _{\T}
&\!\!=\!-\left\langle({\mathbf v}_1\cdot \nabla ){\mathbf v}_3,{\mathbf v}_2\right\rangle _{\T}\quad
 \forall\ {\mathbf v}_1\in \mathbf C^\infty_{\#\sigma},\
{\mathbf v}_2,\, {\mathbf v}_3\in \mathbf C^\infty_\#\,,
\end{align*}
and hence the following well known formula for any  ${\mathbf v}_1\in \mathbf C^\infty_{\#\sigma}$,
${\mathbf v}_2\in \mathbf C^\infty_\#$,
\begin{equation}
\label{eq:mik55}
\left\langle ({\mathbf v}_1\cdot \nabla ){\mathbf v}_2,{\mathbf v}_2\right\rangle _{\T}=0.
\end{equation}

(i) The dense embedding of the space $\mathbf C^\infty_\#$ into ${\mathbf H}_{\#}^{1}$ and into $\mathbf L_{\theta\#}$, the dense embedding of the space $\mathbf C^\infty_{\#\sigma}$ into ${\mathbf H}_{\#\sigma}^{1}$, and estimate \eqref{eq:mik54-2} ensuring the boundedness of the corresponding dual product in \eqref{eq:mik55}, imply that relation \eqref{eq:mik55} holds also for any  
${\mathbf v}_1\in {\mathbf H}_{\#\sigma}^{1}$ and 
${\mathbf v}_2\in {\mathbf H}_{\#}^{1}\cap \mathbf L_{\theta\#}$.

(ii) Particularly, for $n=2$, ${\mathbf H}_{\#}^{1}\cap \mathbf L_{\theta\#}={\mathbf H}_{\#}^{1}\ \forall\ \theta\in (2,\infty)$; for $n\in\{3,4\}$, we can choose $\theta=n$ and take into account that ${\mathbf H}_{\#}^{1}\cap \mathbf L_{n\#}={\mathbf H}_{\#}^{1}$. Hence  if $n\in\{2,3,4\}$,  \eqref{eq:mik55} holds for any  
${\mathbf v}_1\in {\mathbf H}_{\#\sigma}^{1}$ and 
${\mathbf v}_2\in {\mathbf H}_{\#}^{1}$.

(iii) Taking into account also the dense embedding  of $\mathbf C^\infty_{\#\sigma}$ into ${\mathbf L}_{\theta\#\sigma}$, along with estimate \eqref{eq:mik54-2a} ensuring the boundedness of the corresponding dual product in the left hand side of  \eqref{eq:mik55}, we conclude that when  $n\ge 2$, 
\eqref{eq:mik55} holds for any  
${\mathbf v}_1\in {\mathbf L}_{\theta\#\sigma}$ and 
${\mathbf v}_2\in {\mathbf H}_{\#}^{1}$.

\subsubsection{}
Due to Theorem 1 in Section 4.6.1 of \cite{Runst-Sickel1996} 
and equivalence of the Bessel potential norms on square and norms \eqref{eq:mik10} for the Sobolev spaces on torus, we have the following assertion. 
\begin{theorem}\label{th:mik3s1s2}
Let $n\ge 1$, $\tilde s_1\le \tilde s_2$ and $\tilde s_1+\tilde s_2>0$.

(i) If  $\tilde s_2<n/2$
then there exists a constant $C_1=C_1(\tilde s_1, \tilde s_2,n)$ such that for any ${ v}_1\in { H}_{\#}^{\tilde s_1}$ and 
${ v}_2\in { H}_{\#}^{\tilde s_2}$, we have ${ v}_1{ v}_2\in{{ H}_{\#}^{\tilde s_1+\tilde s_2-n/2}}$ and
\begin{align}
\label{eq:mik41-1}
&\left\|{ v}_1{ v}_2\right\|_{{ H}_{\#}^{\tilde s_1+\tilde s_2-n/2}}
\le C_{1}\|{ v}_1\|_{{ H}_{\#}^{\tilde s_1}} \|{ v}_2\|_{{ H}_{\#}^{\tilde s_2}}.
\end{align}

(ii) If $\tilde s_2>n/2$ then  there exists a constant $C_2=C_2(\tilde s_1, \tilde s_2,n)$ such that for any ${ v}_1\in { H}_{\#}^{\tilde s_1}$ and 
${ v}_2\in { H}_{\#}^{\tilde s_2}$, we have ${ v}_1{ v}_2\in{{ H}_{\#}^{\tilde s_1}}$ and
\begin{align}
\label{eq:mik41-2}
&\left\|{ v}_1{ v}_2\right\|_{{ H}_{\#}^{\tilde s_1}}
\le C_{2}\|{ v}_1\|_{{ H}_{\#}^{\tilde s_1}} \|{ v}_2\|_{{ H}_{\#}^{\tilde s_2}}.
\end{align}
\end{theorem}

Theorem \ref{th:mik3s1s2} immediately leads to the following result.
\begin{theorem}\label{th:mik3}
Let $n\ge 2$.

(i) If \, $0<s<n/2$
then the quadratic operators
\begin{align}
\label{eq:mik37-}
&{\bs{\mathit B}}:{\mathbf H}_{\#}^{s}\to {\mathbf H}_{\#}^{2s-1-n/2},\\
\label{eq:mik37}
&{\bs{\mathit B}}:\dot{\mathbf H}_{\#\sigma}^{s}\to \dot{\mathbf H}_{\#}^{2s-1-n/2}
\end{align}
are well defined, continuous and bounded, i.e., there exists $C_{n,s}>0$ such that
\begin{align}
\label{eq:mik38}
&\left\|\bs{\mathit B} {\mathbf w}\right\|_{{\mathbf H}_{\#}^{2s-1-n/2}}
\le  C_{n,s}\|{\mathbf w}\|^2_{{\mathbf H}_{\#}^{s}} 
\quad\forall\ {\mathbf w}\in {\mathbf H}_{\#}^{s}.
\end{align}

(ii) If $s>n/2$ then  the  quadratic operators
\begin{align}\label{eq:mik39-}
&{\bs{\mathit B}}:{\mathbf H}_{\#}^{s}\to {\mathbf H}_{\#}^{s-1},\\
\label{eq:mik39}
&{\bs{\mathit B}}:\dot{\mathbf H}_{\#\sigma}^{s}\to \dot{\mathbf H}_{\#}^{s-1}
\end{align}
are well defined, continuous and bounded, i.e., there exists $C_{n,s}>0$ such that
\begin{align}
\label{eq:mik40}
&\left\|\bs{\mathit B} {\mathbf w}\right\|_{{\mathbf H}_\#^{s-1}}
\le  C_{n,s}\|{\mathbf w}\|^2_{{\mathbf H}_{\#}^{s}} 
\quad\forall\ {\mathbf w}\in {\mathbf H}_{\#}^{s}.
\end{align}
\end{theorem}
\begin{proof}
If a function $\mathbf w$ is periodic, then evidently the function  $\bs{\mathit B}\mathbf w$ is periodic as well.

(i) Let  $0<s<n/2$.
Theorem \ref{th:mik3s1s2}(i) implies estimate \eqref{eq:mik38} and then the boundedness of operator \eqref{eq:mik37-}.

Further, if ${\mathbf u}\in \dot{\mathbf H}_{\#\sigma}^{s}$ then due to the periodicity,
$$
\langle \bs{\mathit B}\mathbf u,1\rangle_{\T}
=\langle {\mathbf u}\cdot \nabla ){\mathbf u} ,1\rangle_{\T}
=-\langle ({\rm div}\, {\mathbf u}){\mathbf u},1 \rangle_{\T}
=\mathbf 0.
$$
Together with estimate \eqref{eq:mik38} this implies that quadratic operator \eqref{eq:mik37} is well defined and bounded.

Let ${\mathbf w},{\mathbf w}'\in {\mathbf H}_{\#\sigma}^{1}$.
Then by  \eqref{eq:mik41-1} we obtain
\begin{align*}
\big\|\bs{\mathit B}{\mathbf w}-\bs{\mathit B}{\mathbf w}'\big\|_{{\mathbf H}_{\#}^{2s-1-n/2}}
&\le\left\|({\mathbf w}\cdot \nabla ){\mathbf w}-({\mathbf w}'\cdot \nabla ){\mathbf w}'\right\|_{{\mathbf H}_{\#}^{2s-1-n/2}}\nonumber\\
&\le\left\|(({\mathbf w}-{\mathbf w}')\cdot \nabla ){\mathbf w} 
+ ({\mathbf w}'\cdot \nabla )({\mathbf w}-{\mathbf w}')\right\|_{{\mathbf H}_{\#}^{2s-1-n/2}}\nonumber\\
&\le  C_{n,s}\left\|{\mathbf w}-{\mathbf w}'\right\|_{{\mathbf H}_{\#}^{s}}\left( \|{\mathbf w}\|_{{\mathbf H}_{\#}^{s}}
+ \|{\mathbf w}'\|_{{\mathbf H}_{\#}^{s}}
\right).
\end{align*}
This estimate shows that operator \eqref{eq:mik37-} and \eqref{eq:mik37} are continuous. 

(ii) Let  $s>n/2$.
Theorem \ref{th:mik3s1s2}(ii) implies estimate \eqref{eq:mik40}. Then by the same arguments as in item (i), one can prove that operators \eqref{eq:mik39-} and \eqref{eq:mik39} are also well defined, bounded and continuous.
\end{proof}

\subsubsection{} 
The following assertion is essentially an adaptation of {\mg a} result from  \cite[Lemma and Remark 2.3(ii)]{Gerhardt1979} to the periodic setting for the particular case of $L_2$-based Sobolev spaces.

\begin{lemma}\label{L6.4}
Let $n\ge 2$ and $\theta\in I_n$.
Then for any $u\in L_{\theta\#}$ and any $\epsilon>0$ there exists a constant $c_\epsilon(u)>0$ depending only on $u$, $n$, and $\epsilon$, such that 
\begin{align}\label{E6.10}
\int_\T|u|^2\,|w|^2 dx\le \epsilon \|w\|^2_{H^1_{\#}}+c_\epsilon(u) \|w\|^2_{L_{2\#}}\quad\forall\ w\in H^1_{\#}.
\end{align}
\end{lemma}
\begin{proof}
Let us {\mg recall} that ${\mg q_\theta}:=2\theta/(\theta-2)>2$, see \eqref{In} and \eqref{2theta}.
Due to the Sobolev embedding theorem, the space $H^1_{\#}$ is continuously embedded in $L_{{\mg q_\theta}\#}$ and hence $w\in L_{{\mg q_\theta}\#}$. 
By the H\"older inequality this implies that $|u|^2\,|w|^2\in L_{1\#}$ and integral in the left hand side of \eqref{E6.10} is bounded.

Let us employ the contradiction argument and assume that there exist $u\in L_{\theta\#}$ and $\epsilon>0$ such that for any constant $c>0$ there exists $w_c$, such that 
\begin{align}\label{E6.11}
\int_\T|u|^2\,|w_c|^2 dx> \epsilon \|w_c\|^2_{H^1_{\#}}+c \|w_c\|^2_{L_{2\#}}.
\end{align}
Inequality \eqref{E6.11} does not hold if $w_c=0$ a.e. on $\T$, hence we can assume that $\|w_c\|_{H^1_{\#}}\ne 0$.
Inequality \eqref{E6.11} then implies that for any $c$ it will also hold for  
$\widetilde w_c=w_c/\|w_c\|_{H^1_{\#}}$. 
We evidently have 
$\|\widetilde w_c\|_{H^1_{\#}}=1$ and by the Sobolev embedding theorem $\|\widetilde w_c\|_{L_{{\mg q_\theta}\#}}$ is bounded by a constant that does not depend on $c$.
By the H\"older inequality for $u\in L_{\theta\#}$ and $\widetilde w_c\in L_{{\mg q_\theta}\#}$,  we have,
\begin{align}\label{E6.12}
\|u\|^2_{L_{\theta\#}}\|\widetilde w_c\|^2_{L_{{\mg q_\theta}\#}}\ge\int_\T|u|^2\,|\widetilde w_c|^2 dx> \epsilon 
+c \|\widetilde w_c\|^2_{L_{2\#}}.
\end{align}
Choosing $c\in\N$ and taking limit of \eqref{E6.12} as $c\to\infty$, the inequality implies that 
\begin{align}\label{E6.12a}
{\mg\| |\widetilde w_c|^2\|_{L_{1\#}}=\|\widetilde w_c\|^2_{L_{2\#}}}\to 0\quad \mbox{as }c\to\infty.
\end{align}
{\mg Note that ${\mg q_\theta}/2>1$.}
The boundedness of  sequence $\|\widetilde w_c\|_{L_{{\mg q_\theta}\#}}$ 
is equivalent to the boundedness of the sequence 
$\| |\widetilde w_c|^2\|_{L_{{\mg q_\theta}/2\#}}$ and
implies that there exists a subsequence of $\{|\widetilde w_c|^2\}$ weakly converging in $L_{{\mg q_\theta}/2\#}$ {\mg (and hence in $L_{1\#}$) to a function $W\in L_{{\mg q_\theta}/2\#}$, 
and then by \eqref{E6.12a}, $W=0$.}
Then the integral in inequality \eqref{E6.12} converges to zero, and the inequality implies that $\epsilon=0$,
which contradicts the 
assumption $\epsilon>0$.
This implies \eqref{E6.10}.
\end{proof}

Let us now prove a stronger version of Lemma~\ref{L6.4}, cf. \cite[Remark 2.3(iii)]{Gerhardt1979}.
\begin{lemma}\label{L6.5}
Let $n\ge 2$  and $\theta\in I_n$.
Let $K_{\theta\#}$ be a compact subset of $L_{\theta\#}$.
Then for 
any $\delta>0$ there exist {\mg constants $\widetilde C_\delta(K_{\theta\#})>0$ and}
$C_\delta(K_{\theta\#})>0$ {\mg(}depending only on $n$, $\theta$, $K_{\theta\#}$ and $\delta${\mg)} such that for all $w\in H^1_{\#}$ and $u\in K_{\theta\#}$,
\begin{align}\label{E6.10an}
\|uw\|_{L_{2\#}}
&\le \delta\|w\|_{H^1_{\#}}+\widetilde C_\delta{\mg(K_{\theta\#})}\|u\|_{L_{\theta\#}}\|w\|_{L_{2\#}}
\\
\label{E6.10b}
&\le \delta \|w\|_{H^1_{\#}}+C_\delta(K_{\theta\#})\|w\|_{L_{2\#}}.
\end{align}
\end{lemma}
\begin{proof}
Let us first prove that for any $\epsilon>0$, there exist a constant $\tilde c_\epsilon>0$, depending only on $n$, $\theta$, and $\epsilon$, such that
\begin{align}\label{E6.10a}
\int_\T|u|^2\,|w|^2 dx
&\le \epsilon \|w\|^2_{H^1_{\#}}+\tilde c_\epsilon\|u\|^2_{L_{\theta\#}}\|w\|^2_{L_{2\#}}.
\end{align}
Let us recall that ${\mg q_\theta}:=2\theta/(\theta-2)>2$, cf. \eqref{In} and \eqref{2theta}.
For $u=0$ or $w=0$  inequality \eqref{E6.10a} evidently holds. 
To prove that it holds also for $\|u\|_{L_{\theta\#}}\ne 0$ and $\|w\|_{H^1_{\#}}\ne 0$ let us assume the contrary, namely, that there exists $\epsilon>0$ such  that for any constant $c>0$ there exists $u_c\in K_{\theta\#}$ and $w_c\in H^1_{\#}$ such that
\begin{align}\label{E6.17}
\int_\T |u_c|^2\,|w_c|^2 dx> \epsilon\|w_c\|^2_{H^1_{\#}}+c\|u_c\|^2_{L_{\theta\#}} \|w_c\|^2_{L_{2\#}}.
\end{align}
Inequality \eqref{E6.17} does not hold if  $w_c=0$, hence we can assume that $\|w_c\|_{H^1_{\#}}\ne 0$.
Inequality \eqref{E6.17} then implies that for any $c$ it will also hold for  
$\widetilde w_c=w_c/\|w_c\|_{H^1_{\#}}$. 
We evidently have 
$\|\widetilde w_c\|_{H^1_{\#}}=1$, and by the Sobolev embedding theorem all $\|\widetilde w_c\|_{L_{{\mg q_\theta}\#}}$ are bounded by a constant that does not depend on $c$.
By the H\"older inequality for $u_c\in L_{\theta\#}$ and $\widetilde w_c\in L_{{\mg q_\theta}\#}$,  inequality \eqref{E6.17} reduces to
\begin{align}\label{E7.32}
\|u_c\|^2_{L_{\theta\#}}\|\widetilde w_c\|^2_{L_{{\mg q_\theta}\#}}\ge\int_\T|u_c|^2\,|\widetilde w_c|^2 dx> \epsilon 
+c \|u_c\|^2_{L_{\theta\#}}\|\widetilde w_c\|^2_{L_{2\#}}.
\end{align}

Let us choose the sequence $\{c\}=\N$. 
Inequality \eqref{E7.32} implies that
\begin{align}\label{E6.19}
\| \tilde w_c\|^2_{L_{{\mg q_\theta}\#}}> c \|\tilde w_c\|^2_{L_{2\#}}.
\end{align}
Since all $\|\widetilde w_c\|_{L_{{\mg q_\theta}\#}}$ are bounded by a constant independent of $c$, inequality \eqref{E6.19}
implies that 
\begin{align}\label{E7.34}
\|\widetilde w_c\|_{L_{2\#}}\to 0\quad \mbox{as }c\to\infty.
\end{align}
The boundedness of  sequence $\|\widetilde w_c\|_{L_{{\mg q_\theta}\#}}$ 
is equivalent to the boundedness of sequence 
$\| |\widetilde w_c|^2\|_{L_{{\mg q_\theta}/2\#}}$ and
implies that there exists a subsequence of $\{|\widetilde w_c|^2\}$ weakly converging in $L_{{\mg q_\theta}/2\#}$ 
{\mg (and hence in $L_{1\#}$) to a function $W\in L_{{\mg q_\theta}/2\#}$, 
and then by \eqref{E7.34}, $W=0$. Here we took into account that ${\mg q_\theta}/2>1$.}
On the other hand, since the set $K_{\theta\#}$ is a compact subset of $L_{\theta\#}$, there exists a subsequence of $\{|u_c|^2\}$ (with the indices $c$ belonging to the subset of indices of the previous  subsequence) strongly converging in $L_{\theta/2\#}$. 

Then (see, e.g., Problem 6.33 in  \cite{Renardy-Rogers1992}) the integral in inequality \eqref{E7.32} converges to zero, and the inequality implies that $\epsilon=0$,
which contradicts the 
assumption $\epsilon>0$.
This implies \eqref{E6.10a}, which  after choosing $\epsilon=\delta^2$ and $\widetilde C_\delta=\sqrt{\tilde c_\epsilon}$ leads to \eqref{E6.10an}.
The set $K_{\theta\#}$ is compact and hence bounded in $L_{\theta\#}$, which then implies inequality \eqref{E6.10b}.
\end{proof}

Note that Lemma \ref{L6.5} particularly holds true when $K_{\theta\#}$ is a sequence converging in $L_{\theta\#}$.

\par\egroup
\end{document}